\documentclass[12pt]{article}

\usepackage{amssymb}
\usepackage{amsthm}
\usepackage{amsmath}
\usepackage{graphicx}
\usepackage{fullpage}
\usepackage{color}
 \numberwithin{equation}{section}
 \usepackage{verbatim}
\usepackage[alphabetic]{amsrefs}
\theoremstyle{plain}
\newtheorem{thm}{Theorem}[section]
\newtheorem{cor}[thm]{Corollary}

\newtheorem{lem}[thm]{Lemma}
\newtheorem{prop}[thm]{Proposition}

\theoremstyle{definition}
\newtheorem{defn}[thm]{Definition}

\theoremstyle{remark}

\newtheorem{rem}[thm]{Remark}

\newcommand{\N}{\mathbb{N}}
\newcommand{\R}{\mathbb{R}}

\newcommand{\Z}{\mathbb{Z}}

\newcommand{\pv}{\text{P.V.}}
\newcommand{\lap}{\Delta}
\newcommand{\e}{\varepsilon}

\newcommand{\I}{\infty}

\newcommand{\bp}{\begin{proof}[\ensuremath{\mathbf{Proof}}]}
\newcommand{\ep}{\end{proof}}
\newcommand{\osc}{\operatorname{osc}}

\newcommand{\J}{{\cal J}_p}

\newcommand{\laps}{(-\lap_p)^s}

\begin{document}

\title{H\"older estimates and large time behavior for a nonlocal doubly nonlinear evolution}

\author{Ryan Hynd\footnote{rhynd@math.upenn.edu, Department of Mathematics, University of Pennsylvania, Philadelphia, Pennsylvania 19104.  Partially supported by NSF grant DMS-1301628.}\; and Erik Lindgren\footnote{eriklin@kth.se, Department of Mathematics, Royal Institute of Technology, 100 44 Stockholm, Sweden. Supported by the Swedish Research Council, grant no. 2012-3124. Partially supported by the Royal Swedish Academy of Sciences. Part of this work was carried out when the second author was visiting University of Pennsylvania. The math department and its facilities are kindly acknowledged.}}  

\maketitle

%  Abstract  
\begin{abstract} \noindent 		The nonlinear and nonlocal PDE
$$
|v_t|^{p-2}v_t+(-\lap_p)^sv=0 \, ,
$$
where
$$
%(-\lap_p)^s u\, (x):=2 \,\pv \int_{\R^n}\frac{|u(x)-u(x+y)|^{p-2}(u(x)-u(x+y))}{|y|^{n+sp}}\, dy,
(-\lap_p)^s v\, (x,t)=2 \,\pv \int_{\R^n}\frac{|v(x,t)-v(x+y,t)|^{p-2}(v(x,t)-v(x+y,t))}{|y|^{n+sp}}\, dy,
$$
has the interesting feature that an associated Rayleigh quotient is non-increasing in time along solutions. We prove the existence of a weak solution of the corresponding initial value problem which is also unique as a viscosity solution. Moreover, we provide H\"older estimates for viscosity solutions and relate the asymptotic behavior of solutions to the eigenvalue problem for the fractional $p$-Laplacian.
\end{abstract}


%%%%%%%%%%%%%%%%%%%%%%%%%%%%%%%%%% Introduction %%%%%%%%%%%%%%%%%%%%%%%%%%%%%%%%%
\section{Introduction}
We study the nonlinear and nonlocal PDE
\begin{equation}\label{MainPDE}
|v_t|^{p-2}v_t+(-\Delta_p)^s v=0
\end{equation}
where $p\in (1,\infty)$, $s\in (0,1)$ and $(-\Delta_p)^s$ is the fractional $p$-Laplacian
\begin{equation}\label{pv}
(-\Delta_p)^s u\, (x):=2\, \pv \int_{\R^n}\frac{|u(x)-u(x+y)|^{p-2}(u(x)-u(x+y))}{|y|^{n+sp}}\, dy.
\end{equation}
Here and throughout $\pv$ denotes principal value. The main reason of our interest in solutions of \eqref{MainPDE} is the connection with ground states for $(-\lap_p)^s$, i.e., extremals of the non-local Rayleigh quotient
\begin{equation}\label{rayleigh}
\lambda_{s,p} = \inf_{u\in W_0^{s,p}(\Omega)\setminus \{0\}}\frac{\displaystyle \int_{\R^n}\int_{\R^n} \frac{|u(x)-u(y)|^p}{|x-y|^{n+sp}} dx dy}{\displaystyle\int_\Omega |u(x)|^p dx}.
\end{equation}
Here and throughout $\Omega\subset \R^n$ is a bounded domain. Clearly, $1/\lambda_{s,p}$ is the optimal constant in the Poincar\'e inequality in the fractional Sobolev space $W_0^{s,p}(\Omega)$. 

\par In recent years there has been a surge of interest around this nonlinear and nonlocal eigenvalue problem, see \cite{LL14}, \cite{BF14}, \cite{BP14}, \cite{BPS15}, \cite{PS14}, \cite{FP14} and \cite{IS14}. In particular, it is known that ground states (or first eigenfunctions) are unique up to a multiplicative constant and have a definite sign (see Theorem 14 in \cite{LL14} together with Corollary 3.14 in \cite{BP14}). The corresponding local problem (formally $s=1$), i.e., the eigenvalue problem for the $p$-Laplacian, has been extensively studied throughout the years. See for instance \cite{Lie83}, \cite{Lin90} and \cite{Lin08}.

\par The first of our main results is a local H\"older estimate for \emph{viscosity solutions} of \eqref{MainPDE}. This is one of the first continuity estimates for parabolic equations involving the fractional $p$-Laplacian.
\begin{thm}\label{Holderthm} Let $p\geq 2$, $s\in (0,1)$ and $v\in L^\infty(\R^n\times (-2,0])$ be a viscosity solution of 
$$|v_t|^{p-2}v_t +(-\lap_p)^s v =0 \text{ in }B_2\times (-2,0].
$$
Then $v$ is H\"older continuous in $B_1\times (-1,0]$ and in particular there exist $\alpha$ and $C$ depending on $p$ and $s$ such that 
$$
\|v\|_{C^\alpha(B_1\times (-1,0])}\leq C\|v\|_{L^\infty(\R^n\times (-2,0])}.
$$
\end{thm}
We also study the initial value problem
\begin{equation}\label{pParabolic}
\begin{cases}
|v_t|^{p-2}v_t+(-\Delta_p)^s v=0,& \Omega\times (0,\infty) \\
v=0, & \R^n\setminus \Omega\times [0,\infty)\\
 v=g,&  \Omega\times\{0\}
\end{cases}
\end{equation}
and show that \eqref{pParabolic} has a {\it weak solution} in the sense of a doubly nonlinear evolution and a unique \emph{viscosity solution}. In addition, we relate the long time behavior of solutions to the eigenvalue problem for the fractional $p$-Laplacian. These results are presented in the two theorems below.
\begin{thm}\label{LargeTthm} Let $p\in (1,\infty)$ and $s\in (0,1)$. Assume $g\in W^{s,p}_0(\Omega)$ and define   
\begin{equation}%\label{mup}
\mu_{s,p}:=\lambda_{s,p}^{\frac{1}{p-1}}.\nonumber
\end{equation}
Then for any weak solution $v$ of \eqref{pParabolic}
\begin{equation}\label{FullTime}
w(x):=\lim_{t\rightarrow \infty}e^{\mu_{s,p} t}v(x,t)
\end{equation}
exists in $W^{s,p}(\R^n)$ and is a ground state for $(-\lap_p)^s$, provided it is not identically zero. In this case, $v(\cdot, t)\neq 0$ for $t\ge 0$ and 
$$
\lambda_{s,p}=\lim_{t\rightarrow \infty}\frac{\displaystyle \int_{\R^n}\int_{\R^n} \frac{|v(x,t)-v(y,t)|^p}{|x-y|^{n+sp}} dx dy}{\displaystyle\int_\Omega|v(x,t)|^pdx}. 
$$
\end{thm}
\begin{thm}\label{LargeTthmvisc} Let $p\geq 2$, $s\in (0,1)$, $\Omega$ be a $C^{1,1}$ domain and assume that $g\in W_0^{s,p}(\Omega)\cap C(\overline\Omega)$ satisfies $|g|\leq \Psi$, where $\Psi$ is a ground state for $(-\lap_p)^s$. Then there is a unique viscosity solution of \eqref{pParabolic} that is also a weak solution. In addition, the convergence in \eqref{FullTime} is uniform in $\overline\Omega$.
\end{thm} 
In our previous work \cite{HL14}, we studied the large time behavior of the doubly nonlinear, local equation 
\begin{equation}
\label{eq:local}
|v_t|^{p-2}v_t = \lap_p v.
\end{equation}
One of the novelties of the present paper in comparison with \cite{HL14}, is that we obtain \emph{uniform} convergence to a ground state and a uniform H\"older estimate for the doubly nonlinear, nonlocal equation \eqref{MainPDE}. No such results are known for equation \eqref{eq:local}. Related to this is also the work for more general systems in \cite{H14}. The method in these papers, as the method in the present paper, differs substantially from most of the other methods used in the literature to study asymptotic behavior of nonlinear and possibly degenerate flows, as in \cite{ABC}, \cite{Aronson}, \cite{AT10}, \cite{KV88}, \cite{KL13}, \cite{SV13}. Our methods are based on energy and compactness in Sobolev spaces while most of the earlier work is based on comparison principles. This allows us, in contrast to most earlier work, to treat initial data without any assumption on the sign.

In the case of a linear equation, i.e., when $p=2$, the large time behavior of solutions is especially well understood. Due to the theory of eigenfunctions in Hilbert spaces one can then recover our result (and more) using the eigenfunction expansion. When $p\neq 2$, this expansion is not available.

The literature on equations of the type \eqref{MainPDE} is very limited. Equations of type  \eqref{eq:local} appears in \cite{KL96} and in the theory of doubly nonlinear flows. In the case of linear non-local equations, i.e., when $p=2$, the literature on regularity is vast. We mention only a fraction, see \cite{Sil12}, \cite{Sil10}, \cite{CV10} and \cite{LD14}. Neither of these results apply to our setting. However, our proof of the H\"older regularity is very much inspired by the work of Luis Silvestre in for instance \cite{Sil12} or \cite{Sil10}. We also seize the opportunity to mention the recent papers \cite{Puh15}, \cite{MRT15}, \cite{Vaz15} and \cite{War16} where the corresponding heat flow is studied, i.e. the equation
$$
v_t +(-\lap_p)^s v=0.
$$
The stationary equation, i.e., 
$$
(-\lap_p)^s v=0,
$$
has in recent years attracted a lot of attention, see \cite{IN10}, \cite{BL15}, \cite{CJ15}, \cite{CKP14}, \cite{CKP13}, \cite{CLM12}, \cite{DKKP15}, \cite{IMS15}, \cite{KMS15}, \cite{KMS15b}, \cite{War16b} and \cite{Lin14}. In \cite{BCF12} a different non-local version of the $p$-Laplacian is studied.

\par The plan of the paper is as follows. In Section \ref{sec:not}, we introduce the fractional Sobolev spaces $W^{s,p}$, the fractional $p$-Laplacian $(-\lap_p)^s$ and additional notation used in this paper. In Section \ref{sec:weak}, we define weak solutions and derive several of their important properties. The section ends with a key compactness result and some brief explanations on how to construct weak solutions. This is followed by Section \ref{sec:visc}, where we introduce viscosity solutions and prove that the weak solution constructed in Section \ref{sec:weak} is also the unique viscosity solution. In Section \ref{sec:holder}, we verify H\"older estimates for viscosity solutions. Finally, in Section \ref{sec:large}, we prove Theorem \ref{LargeTthm} and Theorem \ref{LargeTthmvisc}, which involves the large time behavior
of weak solutions. 
\paragraph{Acknowledgements.} We thank the referees for carefully reviewing this work and providing many useful suggestions.

\section{Notation and prerequisites}\label{sec:not}
The fractional Rayleigh quotient \eqref{rayleigh} naturally relates to the so-called
fractional Sobolev spaces $W^{s,p}(\R^n)$. If $1<p<\infty$ and $s\in (0,1)$ the norm is given by
$$
\|u\|^p_{W^{s,p}(\R^n)}=[u]^p_{W^{s,p}(\R^n)}+\|u\|^p_{L^p(\R^n)},
$$
where the Gagliardo seminorm is 
$$
[u]^p_{W^{s,p}(\R^n)}=\int_{\R^n}\!\!\int_{\R^n} \frac{|u(y)-u(x)|^{p}}{|y-x|^{ sp+n}} d x d y=\int_{\R^n}\!\!\int_{\R^n} |u(y)-u(x)|^{p} d\mu(x,y).
$$
Here and throughout, we will use the notation 
$$
d\mu (x,y):=|x-y|^{-n-sp} dx dy.
$$
The space $W^{s,p}_0(\Omega)$ is the closure of $C_0^\infty(\Omega)$ with respect to the norm $\|\cdot \|_{W^{s,p}(\R^n)}$. Many properties that are known for the more common Sobolev spaces $W^{1,p}$, also hold for $W^{s,p}$ and can be found in \cite{NPV11}. In particular, we have the compact embedding of $W_0^{s,p}(\Omega)$ in $L^q(\Omega)$ for $q\in [1,p]$. This result can be found in Theorem 2.7 in \cite{BLP14} (see also Theorem 7.1 on page 33 in \cite{NPV11}).

The operator $(-\lap_p)^s $ arises as the first variation of the functional 
$$[u]^p_{W^{s,p}(\R^n)}.
$$
More specifically, minimizers satisfy
$$
\int_{\R^n}\int_{\R^n}|u(x)-u(y)|^{p-2}(u(x)-u(y))(\phi(x)-\phi(y))d\mu(x,y) =0,
$$
for each $\phi\in W^{s,p}_0(\Omega)$. If the solution is regular enough, one can split this into two equal terms, make a change of variables and write the equation in the sense of the principal value, as in \eqref{pv}. Note that the notation $(-\Delta_p)^s$ is slightly abusive; this operator is not the $s$th power of $-\lap_p$ unless $p=2$. See Section 3 in \cite{NPV11}.

Ground states of $(-\lap_p)^s$ are minimizers of the  Rayleigh quotient
$$
\lambda_{s,p}=\inf_{u\in W_0^{s,p}(\Omega)}\frac{\displaystyle \int_{\R^n}\int_{\R^n} |u(x)-u(y)|^p d\mu (x,y)}{\displaystyle\int_\Omega |u(x)|^p dx},
$$
and they are signed solutions of
$$
\begin{cases}
\laps u = \lambda_{s,p}|u|^{p-2} u,& \text{in }\Omega,\\
u=0,& \text{in }\R^n\setminus \Omega.
\end{cases}
$$
The notation 
$$
\mathcal{J}_p(t)=|t|^{p-2}t
$$
will also come in handy. With this notation, equation \eqref{MainPDE} can be written as
$$
\mathcal{J}_p(v_t)+\laps v =0
$$
and the operator $\laps$ can be written as
$$
(-\lap_p)^s u(x) = 2\,\pv \int_{\R^n}\frac{\mathcal{J}_p(u(x)-u(x+y))}{|y|^{n+sp}}\, dy.
$$
Due to the scaling of the equation we introduce the following notation for parabolic cylinders
$$
Q_r(x_0,t_0)=B_r(x_0)\times (t_0-r^\frac{sp}{p-1}, t_0+r^\frac{sp}{p-1}),\quad  Q_r^-(x_0,t_0)=B_r(x_0)\times (t_0-r^\frac{sp}{p-1},t_0)
$$
where $B_r(x_0)$ is the ball of radius $r$ centered at $x_0$. When $x_0=0$ and $t_0=0$ we will simply write $B_r$, $Q_r$ and $Q_r^-$.\label{p:def}

%%%%%%%%%%%%%%%%%%%%%%%%%%%%%%%%% Weak Solutions %%%%%%%%%%%%%%%%%%%%%%%%%%%%%
\section{Weak solutions}\label{sec:weak}%\label{WeakSol}
In this section, we present our theory of weak solutions of \eqref{pParabolic}. The main results are that the Rayleigh quotient is monotone along the flow (Proposition \ref{RayleighGoDown}) and that ``bounded'' sequences of weak solutions are compact (Theorem \ref{CompactnessLem}). The interested reader could also consult \cite{HL14} where a similar theory is built for equation \eqref{eq:local}.
% Main identity/Various computations 
\begin{comment}An important identity for smooth solutions of \eqref{pParabolic} is 
\begin{equation}\label{diffEnergyIdentity}
\frac{d}{dt}\int_{\Omega}\frac{1}{p}|Dv(x,t)|^pdx=-\int_{\Omega}|v_t(x,t)|^pdx.
\end{equation}
This identity follows from direct computation.  Of course, integrating \eqref{diffEnergyIdentity} in time yields
\begin{equation}\label{EnergyIdentity}
\int^t_0\int_{\Omega}|v_t(x,s)|^pdxds+\int_{\Omega}\frac{1}{p}|Dv(x,t)|^pdx=\int_{\Omega}\frac{1}{p}|Dg(x)|^pdx
\end{equation}
for $t\ge 0$.  This resulting equality leads us to seek solutions defined as follows. 
% Define weak solutions 
\end{comment}

\begin{defn}%\label{WeakDefn}
Let $g\in W^{s,p}_0(\Omega)$. We say that $v$ is a {\it weak solution} of \eqref{pParabolic} if
\begin{equation}\label{NaturalSpace}
v\in L^\I([0,\infty); W^{s,p}_0(\Omega)), \quad v_t\in L^p(\Omega\times [0,\infty))
\end{equation}
and
\begin{equation}\label{WeakForm}
\int_{\Omega}|v_t(x,t)|^{p-2}v_t(x, t)\phi(x)dx+\int_{\R^n}\int_{\R^n}\J(v(x,t)-v(y,t))(\phi(x)-\phi(y))d\mu(x,y)=0\\
\end{equation}
for each $\phi\in W^{s,p}_0(\Omega)$ and for a.e. $t> 0$, and
\begin{equation}\label{InitialCond}
v(x,0)=g(x). 
\end{equation}
\end{defn}
\begin{rem}\label{rem:ac}
We note that if $v$ satisfies \eqref{NaturalSpace}, the $L^p$ norm of $v$ is absolutely continuous in time (one can for instance adapt the proof of Theorem 3 on page 287 of \cite{LCE}), so that it makes sense to assign values in $L^p(\Omega)$ at $t=0$, as in \eqref{InitialCond}.
\end{rem}
In the rest of this section, we derive various identities and estimates for weak solutions.
\begin{lem}\label{AbsPhi}
Assume $v$ is a weak solution of \eqref{pParabolic}. Then 
$$
[v(\cdot,t)]^p_{W^{s,p}(\R^n)}$$
is absolutely continuous in $t$ and
\begin{equation}\label{diffEnergyIdentity}
\frac{d}{dt} \frac1p [v(\cdot,t)]^p_{W^{s,p}(\R^n)} = -\int_\Omega |v_t(x,t)|^p dx
\end{equation}
holds for almost every $t>0$. 
\end{lem}
\begin{proof} Define  
$$
\Phi(w):=
\begin{cases}
\frac1p [w]^p_{W^{s,p}(\R^n)},& \quad w\in W^{s,p}_0(\Omega)\\
+\infty, &\text{otherwise}
\end{cases}
$$
for each $w\in L^p(\Omega)$. Then $\Phi$ is convex, proper, and lower-semicontinuous. In addition, \eqref{WeakForm} implies
$$
\partial\Phi(v(\cdot,t))=\{-|v_t(\cdot, t)|^{p-2}v_t(\cdot, t)\}
$$
for almost every $t>0$. Since $t\mapsto v(\cdot,t)$ is absolutely continuous in $L^p(\Omega)$ (see Remark \ref{rem:ac}) and since $|\partial\Phi(v)|\cdot |v_t|\in L^1(\Omega\times [0,T))$ for any $T>0$, Remark 1.4.6 in \cite{AGS} implies that $t\mapsto \Phi(v(\cdot,t))$ is absolutely continuous and that identity \eqref{diffEnergyIdentity} holds for a.e. $t>0$.
\end{proof}
\begin{lem}%\label{LemmaINeq}
Assume $v$ is a weak solution of \eqref{pParabolic}. Then
\begin{equation}\label{DiffIneqmu}
\frac{d}{dt}[v(\cdot,t)]^p_{W^{s,p}(\R^n)} \le - p \mu_{s,p} [v(\cdot,t)]^p_{W^{s,p}(\R^n)},
\end{equation}
for a.e. $t>0$, and
\begin{equation}\label{ExpDecay}
[v(\cdot,t)]^p_{W^{s,p}(\R^n)}\le e^{-(p\mu_{s,p})t}[g]_{W^{s,p}(\R^n)}
\end{equation}
for each $t>0$. 

\end{lem}
\begin{proof} By Lemma \ref{AbsPhi}, $v(\cdot,t)$ is an admissible test function in \eqref{WeakForm}, which yields
\begin{align}\label{HolderDerBound}
[v(\cdot,t)]^p_{W^{s,p}(\R^n)}&=\int_{\R^n}\int_{\R^n}\J(v(x,t)-v(y,t))(v(x,t)-v(y,t))d \mu(x,y) \nonumber \\
&=-\int_{\Omega}|v_t(x,t)|^{p-2}v_t(x,t)\cdot v(x,t) dx \nonumber \\
&\le \left(\int_{\Omega}|v_t(x,t)|^pdx\right)^{1-1/p}\left(\int_{\Omega}|v(x,t)|^pdx\right)^{1/p}  \\
&\le \lambda_{s,p}^{-1/p}\left(\int_{\Omega}|v_t(x,t)|^pdx\right)^{1-1/p}[v(\cdot,t)]_{W^{s,p}(\R^n)}.\nonumber
\end{align}
Hence, 
\begin{equation}\label{Dvvtbound}
[v(\cdot,t)]^p_{W^{s,p}(\R^n)}\le \frac{1}{\mu_{s,p}}\int_{\Omega}|v_t(x,t)|^pdx
\end{equation}
Identity \eqref{diffEnergyIdentity} together with \eqref{Dvvtbound} implies \eqref{DiffIneqmu}. From Gr\"{o}nwall's inequality we can now deduce inequality \eqref{ExpDecay}.
\end{proof}

\begin{cor}\label{GradUNon} Let $v$ be a weak solution of \eqref{pParabolic}. Then the function
$$
 e^{(\mu_{s,p} p)t}[v(\cdot,t)]^p_{W^{s,p}(\R^n)}
$$
is nonincreasing in $t$ and 
\begin{equation}\label{eq:needed}
\frac1p\frac{d}{dt}\left( e^{(\mu_{s,p}p)t}[v(\cdot,t)]^p_{W^{s,p}(\R^n)}\right)= e^{(\mu_{s,p}p)t}\left(\mu_{s,p} [v(\cdot,t)]^p_{W^{s,p}(\R^n)}- \int_{\Omega}|v_t(x,t)|^pdx\right)\leq 0,
\end{equation}
for a.e. $t\ge 0.$ 
\end{cor}
\begin{proof}
The monotonicity is a consequence of \eqref{DiffIneqmu}. The identity \eqref{eq:needed} follows from \eqref{diffEnergyIdentity}.
\end{proof}

%prop 
\begin{prop}\label{RayleighGoDown} Assume that $v$ is a weak solution of \eqref{pParabolic} such that $v(\cdot, t)\neq 0\in L^p(\Omega)$ for each 
$t\ge 0$. Then the Rayleigh quotient 
$$
\frac{[v(\cdot,t)]^{p}_{W^{s,p}(\R^n)}}{\displaystyle\int_{\Omega}|v(x,t)|^pdx }
$$
is nonincreasing in $t$.
\end{prop}
\begin{proof}
By \eqref{NaturalSpace}, 
\begin{equation}\label{evanseq}
\frac{d}{dt}\int_\Omega\frac{1}{p}|v(x,t)|^pdx=\int_\Omega|v(x,t)|^{p-2}v(x,t)v_t(x,t)dx
\end{equation}
for a.e. $t>0$. Suppressing the $(x,t)$-dependence, we compute, using \eqref{diffEnergyIdentity} in Lemma \ref{AbsPhi} and \eqref{evanseq}, to find
\begin{align}\label{RayleighComp}
\frac{d}{dt}\frac{[v]^{p}_{W^{s,p}(\R^n)} }{\int_{\Omega}|v|^pdx } & = 
-p\frac{\int_{\Omega}|v_t|^pdx }{\int_{\Omega}|v|^pdx } - p \frac{[v]^{p}_{W^{s,p}(\R^n)} }{\left(\int_{\Omega}|v|^pdx\right)^2} \int_{\Omega}|v|^{p-2} vv_tdx \nonumber \\
&=\frac{p}{\left(\int_{\Omega}|v|^pdx\right)^2}\left\{[v]^{p}_{W^{s,p}(\R^n)}\int_{\Omega}|v|^{p-2} v(-v_t)dx -\int_{\Omega}|v|^pdx \int_{\Omega}|v_t|^pdx \right\}
\end{align}
for a.e. $t>0$.  By H\"{o}lder's inequality 
$$
\int_{\Omega}|v|^{p-2} v(-v_t)dx\le \left(\int_{\Omega}|v|^{p}dx\right)^{1-1/p} \left(\int_{\Omega}|v_t|^{p}dx\right)^{1/p},
$$
which together with \eqref{HolderDerBound} gives
$$
[v]^{p}_{W^{s,p}(\R^n)}\int_{\Omega}|v|^{p-2} v(-v_t)dx \le \int_{\Omega}|v|^pdx \int_{\Omega}|v_t|^pdx.
$$ 
Inserted into \eqref{RayleighComp}, this yields
$$
\frac{d}{dt} \frac{[v]^{p}_{W^{s,p}(\R^n)} }{\displaystyle \int_{\Omega}|v|^pdx }\le 0.
$$
\end{proof}
As a corollary we obtain that any weak solution with a ground state as initial data can be written explicitly. Since the proof is exactly the same as the proof of the corresponding result in \cite{HL14}, Corollary 2.5, we have chosen to omit it.

\begin{cor}\label{cor:icgroundstate} Suppose that $v$ is a weak solution of \eqref{pParabolic} and that $g$ is a ground state of $\laps$. Then 
$$
v(x,t)=e^{-\mu_{s,p}t} g(x).
$$
\end{cor}

The following compactness result is the key both to the long time behavior and to the construction of weak solutions as we will see. The proof is based on the compact embedding of $W_0^{s,p}(\Omega)$ into $L^p(\Omega)$ and it is fairly similar to the proof of Theorem 2.6 in \cite{HL14}.
% Compactness
\begin{thm}\label{CompactnessLem}
Assume $\{g^k\}_{k\in \N}\in W_0^{s,p}(\Omega)$ is uniformly bounded in $W_0^{s,p}(\Omega)$ and that $v^k$ is a weak solution of \eqref{pParabolic} with $v^k(x,0)=g^k(x)$. Then there is a subsequence $\{v^{k_j}\}_{j\in \N}\subset \{v^k\}_{k\in \N}$ and $v$ satisfying 
\eqref{NaturalSpace} such that 
\begin{equation}\label{FirstConv}
v^{k_j}\rightarrow v\quad \text{in}\quad
\begin{cases}
C([0,T]; L^p(\Omega)), \quad \text{ for all } \quad T>0\\
L^r_{\textup{loc}}([0,\infty); W^{s,p}(\R^n)), \quad \text{ for all }\quad 1\le r<\infty
\end{cases}
\end{equation}
and
\begin{equation}\label{SecondConv}
v_t^{k_j}\rightarrow v_t\quad \text{in}\quad L^p_{\textup{loc}}([0,\infty); L^p(\Omega))
\end{equation}
as $j\rightarrow \infty$.  Moreover, $v$ is a weak solution of \eqref{pParabolic} where $g$ is a weak limit of $\{g^{k_j}\}_{k\in \N}$ in $W_0^{s,p}(\Omega)$. 
\end{thm}
\begin{proof}
As in \eqref{diffEnergyIdentity}, 
\begin{equation}\label{Energyk}
\frac{d}{dt}\frac1p [v^k(\cdot,t)]^p_{W^{s,p}(\R^n)}=-\int_{\Omega}|v^k_t(x,t)|^pdx,
\end{equation}
for almost every $t>0$. After integration we obtain
\begin{equation}\label{EnergykBound}
p \int^\infty_0\int_\Omega |v^k_t(x,t)|^pdxdt +\sup_{t\ge 0}[v_k(\cdot,t)]_{W^{s,p}(\R^n)} \le 2[g^k]^p_{W^{s,p}(\R^n)}.
\end{equation}
By assumption, the right hand side above is uniformly bounded. It follows that the sequence $\{v^k\}_{k\in \N}\in C_{\text{loc}}([0,\infty),L^p(\Omega))$ is equicontinuous, and $\{v^k(\cdot,t)\}_{k\in \N}\in W^{s,p}_0(\Omega)$ is uniformly bounded for each $t\ge 0$. By Theorem 1 in \cite{Sim87}, we can conclude that there is a subsequence $\{v^{k_j}\}_{j\in \N}\subset \{v^k\}_{k\in \N}$ converging in $C_{\text{loc}}([0,\infty),L^p(\Omega))$ to some $v$ satisfying \eqref{NaturalSpace}.  Passing to a further a subsequence, we may also assume that $v^{k_j}\rightharpoonup v$ in  $L^p_{\text{loc}}([0,\infty); W^{s,p}(\R^n))$. 

\par Since $v_t^{k}$ is bounded in $L^p(\Omega\times [0,\infty))$,  we may also assume 
$$
\begin{cases}
v_t^{k_j}\rightharpoonup v_t\quad \text{in}\quad L^p(\Omega\times [0,\infty))\\
{\cal J}_p(v_t^{k_j})\rightharpoonup \xi\quad \text{in}\quad L^q(\Omega\times [0,\infty))
\end{cases}
$$
where $1/p+1/q=1$. We will prove below that
\begin{equation}\label{needWeak}
\xi={\cal J}_p(v_t).
\end{equation}
Let us assume for the moment that \eqref{needWeak} holds.  Note that since the function $|z|^{p}$ is convex,
$$\frac1p [w]^p_{W^{s,p}(\R^n)}
\ge  \frac1p [v^{k_j}(\cdot,t)]^p_{W^{s,p}(\R^n)} -\int_{\Omega}{\cal J}_p(v_t^{k_j}(x,t))(w(x)-v^{k_j}(x,t))dx
$$
for any $w\in W^{s,p}_0(\Omega)$. Integrating over the interval $ [t_0, t_1]$ and passing to the limit, we obtain 
$$
\int^{t_1}_{t_0}\frac1p [w]^p_{W^{s,p}(\R^n)}dt \ge  \int^{t_1}_{t_0}\left(\frac1p [v(\cdot,t)]^p_{W^{s,p}(\R^n)} -\int_{\Omega}\xi(x,t)(w(x)-v(x,t))dx\right)dt. 
$$
Here we made use of Fatou's lemma, the weak convergence of ${\cal J}_p (v_t^{k_j})$ and  the strong convergence of $v^{k_j}$.

Therefore, 
$$
\frac1p [w]^p_{W^{s,p}(\R^n)}\ge\frac1p [v(\cdot,t)]^p_{W^{s,p}(\R^n)} -\int_{\Omega}\xi(x,t)(w(x)-v(x,t))dx
$$
for a.e. $t\ge 0$. In particular, for each $\phi \in W^{s,p}_0(\Omega)$
\begin{equation}\label{haveWeak}
%\xi=\Delta_pv.
\int_{\Omega}\xi(x, t)\phi(x)dx +\int_{\R^n}\int_{\R^n}\J(v(x,t)-v(y,t))(\phi(x)-\phi(y)) d\mu(x,y)=0
\end{equation}
for a.e. $t\ge 0$.  Thus, once we verify \eqref{needWeak}, $v$ is then a weak solution of \eqref{pParabolic}.  

For each interval $[t_0, t_1]$
\begin{align*}
&\lim_{j\rightarrow \infty}\int^{t_1}_{t_0} [v^{k_j}(\cdot,t)]^p_{W^{s,p}(\R^n)} dt  \\
&=\lim_{j\rightarrow \infty}\int^{t_1}_{t_0} \int_{\R^n}\int_{\R^n}\J(v^{k_j}(x,t)-v^{k_j}(y,t))(v^{k_j}(x,t)-v^{k_j}(y,t)) d\mu(x,y)\\
&=-\lim_{j\rightarrow \infty}\int^{t_1}_{t_0}\int_\Omega {\cal J}_p(v_t^{k_j}(x,t)) v^{k_j}(x,t)dxdt\\
&=-\int^{t_1}_{t_0}\int_\Omega \xi(x,t) v(x,t)dxdt\\
&=\int^{t_1}_{t_0}[v(\cdot,t)]^p_{W^{s,p}(\R^n)} dt,
\end{align*}
where the last equality is a consequence of \eqref{haveWeak}. Since weak convergence together with convergence of the norm implies strong convergence, we have
$$
v^{k_j}\rightarrow v\quad \text{in}\quad L^p_{\text{loc}}([0,\infty); W^{s,p}(\R^n)).
$$
\par It is now routine to combine the interpolation of $L^p$ spaces with the uniform bound \eqref{EnergykBound} to obtain the stronger convergence $v^{k_j}\rightarrow v$ in $L^r_{\text{loc}}([0,\infty); W^{s,p}(\R^n))$ for each $1\le r<\infty$. Further, upon extracting yet another subsequence, we can assume that
\begin{equation}\label{StrongConverge}
[v^{k_j}(\cdot,t)]^p_{W^{s,p}(\R^n)}\rightarrow [v(\cdot,t)]^p_{W^{s,p}(\R^n)}
\end{equation}
as $j\rightarrow \infty$, for a.e. $t\ge 0$.

\par We will now verify \eqref{needWeak}. As in the proof of Lemma \ref{AbsPhi}, \eqref{haveWeak} implies 
$$
\frac{d}{dt} \frac{1}{p}[v(\cdot,t)]^p_{W^{s,p}(\R^n)}= -\int_{\Omega}\xi(x,t)v_t(x,t)dx,
$$
for a.e. $t\geq 0$. Therefore, for each $t_1> t_0$
\begin{equation}\label{NeedtoCompare}
\int^{t_1}_{t_0}\int_{\Omega}\xi(x,s)v_t(x,s)dxds+\frac{1}{p}[v(\cdot,t_1)]^p_{W^{s,p}(\R^n)}=\frac{1}{p}[v(\cdot,t_0)]^p_{W^{s,p}(\R^n)}.
\end{equation}
In addition, integrating \eqref{Energyk} yields
\begin{equation}\label{NeedtoCompare2}
\int^{t_1}_{t_0}\int_{\Omega}\frac{1}{p}|v^{k_j}_t(x,s)|^p+\frac{1}{q}|{\cal J}_p(v_t^{k_j}(x,s))|^qdxds+\frac{1}{p}[v^{k_j}(\cdot,t_1)]^p_{W^{s,p}(\R^n)}= \frac{1}{p}[v^{k_j}(\cdot,t_0)]^p_{W^{s,p}(\R^n)}.
\end{equation}

Let now $t_0$ and $t_1$ be times for which \eqref{StrongConverge} holds and pass to the limit to obtain
$$
\int^{t_1}_{t_0}\int_{\Omega}\frac{1}{p}|v_t(x,s)|^p+\frac{1}{q}|\xi(x,s)|^qdxds+\frac{1}{p}[v(\cdot,t_1)]^p_{W^{s,p}(\R^n)}\le \frac{1}{p}[v(\cdot,t_0)]^p_{W^{s,p}(\R^n)}
$$
by weak convergence.  Together with \eqref{NeedtoCompare} this implies
$$
\int^{t_1}_{t_0}\int_{\Omega}\left(\frac{1}{p}|v_t(x,s)|^p+\frac{1}{q}|\xi(x,s)|^q -  \xi(x,s)v_t(x,s)\right)dxds\le 0.
$$
Identity \eqref{needWeak} now follows from the case of equality in Young's inequality.

Substituting $\xi={\cal J}_p(v_t)$ into \eqref{NeedtoCompare} and passing to the limit as $j\rightarrow \infty$ in \eqref{NeedtoCompare2} also gives 
$$
\lim_{j\rightarrow \infty}\int^{t_1}_{t_0}\int_{\Omega}|v^{k_j}_t(x,s)|^pdxds=\int^{t_1}_{t_0}\int_{\Omega}|v_t(x,s)|^pdxds.
$$
Again, since weak convergence together with convergence of the norm implies strong convergence, we obtain \eqref{SecondConv}.
\end{proof}

% Sketch implicit scheme
Let us now discuss how the ideas above can be used to construct weak solutions. As in \cite{HL14}, we aim to build weak solutions \eqref{pParabolic} by using the implicit time scheme for $\tau>0$: $v^0:=g$, 
\begin{equation}
\begin{cases}\label{ViscScheme}
{\cal J}_p\left(\frac{v^k - v^{k-1}}{\tau}\right)+(-\Delta_p)^s v^k=0,& x\in \Omega \\
 v^k=0, & x\in\R^n\setminus  \Omega
\end{cases}\quad k=1,2,\dots, N.
\end{equation}
The direct methods in the calculus of variations can be used to show that this scheme has a unique weak solution sequence $\{v^1,\dots, v^N\}\subset W^{s,p}_0(\Omega)$ for each $\tau>0$ and $N$, in the sense that
$$
\int_\Omega {\cal J}_p\left(\frac{v^k(x) - v^{k-1}(x)}{\tau}\right)\phi(x) dx+\int_{\R^n}\int_{\R^n} \J(v^k(x)-v^k(y))(\phi(x)-\phi(y))d \mu(x,y)=0, 
$$
for any $\phi\in W_0^{s,p}(\Omega)$. Our candidate for a solution $v(x,t)$ of \eqref{pParabolic} is the limit of $v^{N}(x)$, when $N$ tends to infinity with $\tau=t/N$. 

\par Choosing $\phi=v^k-v^{k-1}$ as test function, we 
obtain 
$$
%\frac{\left|u_k-u_{k-1}\right|^p}{\tau^{p-1}} + \Phi(u_k)\le \Phi(u_{k-1}).
\int_{\Omega}\frac{|v^k-v^{k-1}|^p}{\tau^{p-1}}dx +\frac{1}{p}[v^{k}]^p_{W^{s,p}(\R^n)}\le \frac{1}{p}[v^{k-1}]^p_{W^{s,p}(\R^n)}, \quad k=1, \dots, N.
$$
Summing over $k=1,\dots, j\le N$ yields
\begin{equation}\label{DiscIden}
\sum^{j}_{k=1}\int_{\Omega}\frac{|v^k-v^{k-1}|^p}{\tau^{p-1}}dx + \frac{1}{p}[v^{j}]^p_{W^{s,p}(\R^n)}\le \frac{1}{p}[g]^p_{W^{s,p}(\R^n)},
\end{equation}
which is a discrete analogue of the energy identity \eqref{diffEnergyIdentity}.

Let $\tau=T/N$ and $\tau_k=k\tau$, and define the ``linear interpolation" of the solution sequence as
\begin{equation}\label{LineApprox}
w_N(x,t):=
v^{k-1}(x)+\left(\frac{t-\tau_{k-1}}{\tau}\right)(v^k(x)-v^{k-1}(x)), 
\quad \tau_{k-1}\le t\le \tau_k, \quad k=1,\dots, N.
\end{equation}
From \eqref{DiscIden} we conclude
\begin{equation}
p\int^T_0\int_\Omega |\partial_t w_N(x,t)|^pdxdt + \sup_{0\le t\le T}[w_N(\cdot,t)]^p_{W^{s,p}(\R^n)} \le 2[g]^p_{W^{s,p}(\R^n)} \nonumber 
\end{equation}
Arguing as in the proof of Theorem \ref{CompactnessLem}, we can obtain a subsequence $w_{N_j}$ and a weak solution $w$ of \eqref{MainPDE} on $\Omega\times (0,T)$ such that
$$
w_{N_j}\rightarrow w
\quad \text{in}\quad \begin{cases}
C([0,T]; L^p(\Omega))\\
L^p([0,T]; W^{s,p}(\R^n))
\end{cases}
$$
and
$$
\begin{cases}
\partial_tw_{N_j}\rightarrow w_t\quad \text{in}\quad L^p(\Omega\times [0,T]),\\
{\cal J}_p(\partial_tw_{N_j})\rightharpoonup {\cal J}_p(w_t)\quad \text{in}\quad L^q(\Omega\times [0,T]).
\end{cases}
$$

\par It remains to construct a global in time solution. This can be accomplished as follows: Let $k\in \N$ and let $w^k$ be the weak solution of \eqref{MainPDE} above for $T=k$. Define
$$
z^k(\cdot, t)=\begin{cases} w^k(\cdot, t), & t\in [0,k],\\
w^k(\cdot, k), & t\in [k,\infty).
\end{cases}
$$
One readily verifies that $z^k$ satisfies \eqref{NaturalSpace}. In addition, the proof of Theorem \ref{CompactnessLem} can easily be adapted to give that $z^k$ has a subsequence converging as in \eqref{FirstConv}, \eqref{SecondConv} to a global weak solution of \eqref{pParabolic}. We omit the details. 

\begin{rem} \label{rem:diffeq} At this point, we seize the opportunity to mention that the ``step function" approximation 
\begin{equation}\label{StepApprox}
v_N(\cdot,t):=
\begin{cases}
g, & t=0\\
v^k, & t\in (\tau_{k-1},\tau_k], \quad k=1,\ldots,N
\end{cases}
\end{equation}
converges in $C([0,T]; L^p(\Omega))$\ to the same weak solution $v$ as the linear interpolating sequence \eqref{LineApprox}. Indeed, by \eqref{DiscIden}
\begin{align*}
\int_{\Omega}|w_N(x,t)-v_N(x,t)|^pdx &\le \max_{1\le k\le N}\int_{\Omega}|v^k(x)-v^{k-1}(x)|^pdx \\
& \le \frac1p \tau^{p-1}[g]^p_{W^{s,p}(\R^n)} \\
&= \frac1p\left(\frac{T}{N}\right)^{p-1}[g]^p_{W^{s,p}(\R^n)}.
\end{align*}
%This fact will be used in Section \ref{sec:visc}, when we identify that our constructed viscosity solution is the same as the constructed weak solution.
This fact will be used in Section \ref{sec:visc}, where we verify that the viscosity solution we construct is also a weak solution.
\end{rem}

%%%%%%%%%%%%%%%%%%%%%%%%%%%%%%%%% Viscosity solutions %%%%%%%%%%%%%%%%%%%%%%%%%%%%%
\section{Viscosity solutions}\label{sec:visc}%\label{ViscSoln}

Throughout this section we assume that $\partial \Omega$ is $C^{1,1}$, $p\ge 2$ , $g\in W_0^{s,p}(\Omega)\cap C(\overline\Omega)$ and that there is a ground state $\Psi$ such that 
$$
|g|\leq \Psi.
$$ 
Our main result in this section is:

 \begin{prop}\label{ViscSolnResult}
There is a unique viscosity solution $v$ of the initial value problem \eqref{pParabolic} which is also a weak solution.
\end{prop}
% uniqueness of viscosity solutions
It is not known whether or not uniqueness holds for weak solutions of \eqref{pParabolic}, even in the local case. However, quite standard methods for viscosity solutions apply to \eqref{pParabolic}. The key here is that the term $|v_t|^{p-2}v_t$ is strictly monotone with respect to $v_t$. In what follows, we will prove that the discrete scheme \eqref{ViscScheme} converges both to the unique viscosity solution and to a weak solution.

We first define \emph{viscosity solutions} of the relevant equations.
\begin{defn}%\label{ellipticvisc}
Let $\Omega$ be an open set in $\R^n$ and $f(x,u)$ a continuous function. A function $u\in L^\infty(\R^n)$ which is upper semicontinuous in ${\Omega}$ is a \emph{subsolution} of 
$$  
(-\lap_p)^s u\,\leq f(x,u) \text{ in $\Omega$}
$$
if the following holds: whenever $x_0\in \Omega$ and $\phi\in C^2({B_r(x_0)})$ for some $r>0$ are such that
$$
\phi(x_0)=u(x_0), \quad \phi(x)\geq u(x) \text{ for $x\in B_r(x_0)\subset \Omega$}
$$
then we have
$$
(-\lap_p)^s\phi^r\, (x_0)\leq f(x_0,\phi(x_0)),
$$
where
$$
\phi^r =\left\{\begin{array}{lr}\phi \text{ in }B_r(x_0),\\
u\text{ in }\R^n\setminus B_r(x_0).
\end{array}\right. 
$$
A \emph{supersolution} is defined similarly and a \emph{solution} is a function which is both a sub- and a supersolution.
\end{defn}

\begin{rem}\label{welldef} For a bounded function $f$ which is $C^2$ in a neighborhood of $x_0$, $(-\lap_p)^s f\, (x_0)$ is well defined. Indeed, we may split the operator into one integral over $B_1(x_0)$ and another over $\R^n\setminus B_1(x_0)$. The latter is well-defined since $f$ is bounded. For the former, we may write for $\e>0$
\begin{align*}
&2\pv \int_{B_1(x_0)}\J(f(x_0)-f(y))|x_0-y|^{-n-sp} dy\\
&=2\lim_{\e\to 0}\int_{B_1\setminus B_\e}\J(f(x_0)-f(x_0+y))|y|^{-n-sp} dy\\
&=\lim_{\e\to 0}\int_{B_1\setminus B_\e}\left(\J(f(x_0)-f(x_0+y))+\J(f(x_0)-f(x_0-y))\right)|y|^{-n-sp} dy.
\end{align*}
Since $f$ is a $C^2$ function and $p\geq 2$, we have the estimate
$$
|f(x)-f(x+y)|^{p-2}(f(x)-f(x+y))+|f(x)-f(x-y)|^{p-2}(f(x)-f(x-y))\leq C|y|^p.
$$
This is due to the elementary inequality
$$
\Big||a+b|^{p-2}(a+b)-|a|^{p-2}a\Big|\leq C |b|(|a|+|b|)^{p-2}.
$$
Therefore, the integral 
$$
\int_{B_1}\left(\J(f(x_0)-f(x_0+y))+\J(f(x_0)-f(x_0-y))\right)|y|^{-n-sp} dy
$$
is absolutely convergent, so that the principal value exists.
\end{rem}

We define viscosity solutions of the evolutionary equation \eqref{MainPDE}. We introduce the class of test functions
$$
C^{2,1}_{x,t}(D\times I), \quad D\times I\subset \R^n\times \R.
$$
consisting of functions that are $C^2$ in the spatial variables and $C^1$ in $t$, in the set $D\times I$.

\begin{defn}\label{paravisc} Let $\Omega \in \R^n$ be an open set, $I\in \R$ be an open interval. A function $v\in L^\infty(\R^n\times I)$ which is upper semicontinuous in ${\Omega\times I}$ is a \emph{subsolution} of 
$$  
|v_t|^{p-2}v_t+(-\lap_p)^s v\,\leq C \text{ in $\Omega \times I$}
$$
if the following holds: whenever $(x_0,t_0)\in \Omega\times I$ and $\phi\in C^{2,1}_{x,t}({B_r(x_0)}\times (t_0-r,t_0+r))$ for some $r>0$ are such that
$$
\phi(x_0,t_0)=v(x_0,t_0), \quad \phi(x,t)\geq v(x,t) \text{ for $(x,t)\in B_r(x_0)\times (t_0-r,t_0+r)$}
$$
then
$$
|\phi_t(x_0,t_0) |^{p-2}\phi_t(x_0,t_0)+(-\lap_p)^s\phi^r\, (x_0,t_0)\leq C,
$$
where
$$
\phi^r (x,t) =\left\{\begin{array}{lr}\phi \text{ in }B_r(x_0)\times (t_0-r,t_0+r),\\
v\text{ in }\R^n\setminus B_r(x_0)\times (t_0-r,t_0+r).
\end{array}\right. 
$$
A \emph{supersolution} is defined similarly and a \emph{solution} is a function which is both a sub- and a supersolution.
\end{defn}
\begin{rem}
In both of the definitions above, it is obvious that we can replace the condition that $\phi$ touches $v$ from above at a point, with the condition that $v-\phi$ has a maximum at a point. In addition, as is standard when dealing with viscosity solutions, it is enough to ask that $\phi$ touches $v$ strictly at a point or equivalently that $v-\phi$ has a strict maximum at a point.
\end{rem}
Now we are ready to treat the implicit scheme \eqref{ViscScheme}. We first construct viscosity solutions $v^1,\dots, v^N$.

\begin{lem}\label{SchemeVISC} For each $N$ and $\tau$,  the implicit scheme \eqref{ViscScheme} generates viscosity solutions, $v^k\in C(\overline \Omega)$ for $k=1,\ldots,N$. Moreover,
$$
\max_{1\le k\le N}\sup_{\R^n}|v^k|\le\sup_{\R^n}|g|.
$$
\end{lem}
\begin{proof}
Consider the implicit scheme \eqref{ViscScheme} for $k=1$
$$
{\cal J}_p\left(\frac{v^1 - g}{\tau}\right)+(-\Delta_p)^s v^1= 0,\quad x\in \Omega.
$$
This means that
$$
\int_\Omega {\cal J}_p\left(\frac{v^1 - g}{\tau}\right)\phi dx+\int_{\R^n}\int_{\R^n} \J(v^1(x)-v^1(y))(\phi(x)-\phi(y))d \mu(x,y)=0
$$
for any $\phi\in W_0^{s,p}(\Omega)$. The existence of such a weak solution follows from the direct methods of calculus of variations. Since ${\cal J}_p$ is strictly increasing it is standard to prove a comparison principle for weak sub- and supersolutions, see for instance Lemma 9 in \cite{LL14} for a proof. Clearly, the constant function $\sup_{\R^n}|g|$ is a supersolution, hence $v^1\le \sup_{\R^n}|g|$. Similarly, 
$v^1\ge -\sup_{\R^n}|g|$, and thus $|v^1|\le \sup_{\R^n}|g|$. By induction, $|v^k|\le \sup_{\R^n}|g|$ for $k=2,\dots, N$. As the left hand side of the PDE \eqref{ViscScheme} is bounded, it follows by Theorem 1.1 in \cite{IMS15}, that $v^k$ is continuous in $\overline\Omega$ for $k=1,\dots, N$.

That each $v^k$ is a viscosity solution can be verified by following the proof of Proposition 11 in \cite{LL14} line by line. We omit the details.
\end{proof}

The natural candidate for a viscosity solution of \eqref{pParabolic} is $\lim_{N\rightarrow\infty} v_N$ where $v_N$ is defined in \eqref{StepApprox}.  Before proving this, we present some technical lemmas.
% Discrete viscosity property 
\begin{lem}\label{LemdiscreteVisc}
Let $N\in \N$. Further assume $\{\psi^0,\psi^1,\dots,\psi^N\}\subset C^2(\Omega)$ and $(x_0,k_0)\in\Omega\times\{1,\dots,N\}$ are such that
\begin{equation}\label{discreteVisc}
v^k(x)-\psi^k(x)\le v^{k_0}(x_0)-\psi^{k_0}(x_0) 
\end{equation}
for $x$ in $B_r(x_0)$ and $k\in \{k_0-1,k_0\}$. Then 
$$
{\cal J}_p\left(\frac{\psi^{k_0}(x_0) - \psi^{k_0-1}(x_0)}{\tau}\right)+(-\Delta)_p^s(\psi^{k_0})^r(x_0)\leq 0.
$$
\end{lem}

\begin{proof}
Evaluating the left hand side \eqref{discreteVisc} at $k=k_0$ gives 
$$
{\cal J}_p\left(\frac{v^{k_0}(x_0) - v^{k_0-1}(x_0)}{\tau}\right)+(-\Delta_p)^s (\psi^{k_0})^r(x_0)\leq 0
$$
as $v^k$ is a viscosity solution of \eqref{ViscScheme}.  Evaluating the left hand side of \eqref{discreteVisc} at $x=x_0$ and $k=k_0-1$ gives
$\psi^{k_0}(x_0) - \psi^{k_0-1}(x_0)\le v^{k_0}(x_0) - v^{k_0-1}(x_0)$. The claim follows from the above inequality and the monotonicity of 
${\cal J}_p$. 
\end{proof} 
Let $\overline{v}$ and $\underline{v}$ denote the weak upper and lower limits respectively, of $v_N$ defined in \eqref{StepApprox}, i.e., 
$$
\overline{v}(x,t):=\limsup_{\substack{N\rightarrow\infty\\ (y,s)\rightarrow(x,t)}}v_N(y,s),
$$
$$
\underline{v}(x,t):=\liminf_{\substack{N\rightarrow\infty\\ (v,s)\rightarrow(x,t)}}v_N(y,s).
$$
By Lemma \ref{SchemeVISC}, the sequence $\{v_N\}_{N\in \N}$ is bounded  independently of $N\in \N$.  As a result, $\overline{v}$ and $\underline{v}$ are well defined and finite. In addition, $\overline{v}$ and $-\underline{v}$ are upper semicontinuous. We recall the following result, which is Lemma 4.4 in \cite{HL14}. The statement of the result in \cite{HL14} is for smooth $\phi$, but the proof holds also for $\phi$ as below.
\begin{lem}\label{ApproxPoints}
Assume $\phi\in C^{2,1}_{x,t}((\Omega\times(0,T))$.  For $N\in \N$ define 
$$
\phi_N(x,t):=
\begin{cases}
\phi(x,0),& t=0\\
\phi(x,\tau_k),& t\in (\tau_{k-1},\tau_k],\quad  k=1,\ldots,N.
\end{cases}
$$
Suppose $\overline{v}-\phi$ has a strict local maximum at $(x_0,t_0)\in \Omega\times(0,T)$. Then there is $(x_j,t_j)\rightarrow  (x_0,t_0)$ and $N_j\rightarrow \infty$, as $j\rightarrow \infty$, such that 
$v_{N_j} -\phi_{N_j}$ has local maximum at $(x_j,t_j)$. A corresponding result holds in the case of a strict local minimum.
\end{lem}
Before proving the uniqueness of viscosity solutions we need the following result, which verifies that whenever we can touch a subsolution from above with a $C_{x,t}^{2,1}$ function, we can treat the subsolution as a classical subsolution in space. The proof is almost identical to the one of Theorem 2.2 in \cite{CS09} or the one of Proposition 1 in  \cite{Lin14}.
\begin{prop}\label{pw}  Suppose 
$$
|v_t|^{p-2}v_t+(-\lap_p)^s v\leq C \text{ in $B_1\times I$}
$$
in the viscosity sense. Further assume that $(x_0,t_0)\in B_1\times I$ and $\phi\in C^{2,1}_{x,t}({B_r(x_0)}\times (t_0-r,t_0+r))$ are such that
$$
\phi(x_0,t_0)=v(x_0,t_0), \quad \phi(x,t)\geq v(x,t)\quad  \text{ for $(x,t)\in B_r(x_0)\times (t_0-r,t_0+r)$}
$$
for some $r>0$. Then $(-\lap_p)^s v$ is defined pointwise at $(x_0,t_0)$ and 
$$
|\phi_t(x_0,t_0) |^{p-2}\phi_t(x_0,t_0)+(-\lap_p)^s v\, (x_0,t_0)\leq C.  
$$
\end{prop}
\begin{proof} For $0<\rho \leq r$, let
 $$
 \phi^\rho=\left\{\begin{array}{lr} \phi \text{ in }B_\rho(x_0)\times (t_0-r,t_0+r),\\
v\text{ in }\R^n\setminus B_\rho(x_0)\times (t_0-r,t_0+r).
 \end{array}\right.
 $$
 Since $v$ is a viscosity subsolution,
 $$
|\phi_t(x_0,t_0) |^{p-2}\phi_t(x_0,t_0)+(-\lap_p)^s\phi^\rho\, (x_0,t_0)\leq C.
$$
 \par Now introduce the notation 
 \begin{align*}
\delta(\phi^\rho, x,y,t):=&\frac{1}{2}|\phi^\rho(x,t)-\phi^\rho(x+y,t)|^{p-2}(\phi^\rho(x,t)-\phi^\rho(x+y,t))\\&+\frac{1}{2}|\phi^\rho(x,t)-\phi^\rho(x-y,t)|^{p-2}(\phi^\rho(x,t)-\phi^\rho(x-y,t)),
\end{align*}
$$
 \delta^\pm(\phi^\rho,x,y,t) = \max(\pm \delta(\phi^\rho,x,y,t),0).
$$
Since $\phi^\rho$ is $C^2$ in space near $x_0$, we can substitute $-y$ for $y$ in the integral and obtain the convergent integral
\begin{equation}\label{eq:deltasubsol}
2 \int_{\R^n}\delta(\phi^\rho,x_0,y,t_0)|y|^{-n-sp}\, dy\leq C-|\phi_t(x_0,t_0) |^{p-2}\phi_t(x_0,t_0):=D
\end{equation}
See Remark \ref{welldef} for more details.
%Indeed, if $f(x)$ is a $C^2$ function, then 
%$$
%|f(x)-f(x+y)|^{p-2}(f(x)-f(x+y))+|f(x)-f(x-y)|^{p-2}(f(x)-f(x-y))=\mathcal{O}(|y|%^p), 
%$$
%when $p\geq 2$.
\par We note that 
$$\delta(\phi^{\rho_2},x_0,y,t_0)\leq \delta(\phi^{\rho_1},x_0,y,t_0)\leq \delta(v,x_0,y,t_0)\text{ for }\rho_1< \rho_2< r,
$$
so that 
\begin{equation}\label{deltaminus}
\delta^-(\phi^{\rho_2},x_0,y,t_0)\geq \delta^-(\phi^{\rho_1},x_0,y,t_0)\geq \delta^-(v,x_0,y,t_0)\text{ for }\rho_1< \rho_2< r.
\end{equation}
In particular, 
$$
\delta^-(v,x_0,y,t_0)\leq |\delta(\phi^r,x_0,y,t_0)|.
$$
Since $|\delta(\phi^r,x_0,y,t_0)|y|^{-n-sp}|$ is integrable, so is $\delta^-(v,x_0,y,t_0)|y|^{-n-sp}$. In addition, by \eqref{eq:deltasubsol}
$$
2\int_{\R^n}\delta^+(\phi^\rho,x_0,y,t_0)|y|^{-n-sp}\, dy\leq 2\int_{\R^n}\delta^-(\phi^\rho,x_0,y,t_0)|y|^{-n-sp}\, dy+D.
$$
Thus, for $\rho_1<\rho_2$
\begin{align}\label{eq:srineq}
2\int_{\R^n}\delta^+(\phi^{\rho_1},x_0,y,t_0)|y|^{-n-sp}\, dy&\leq2\int_{\R^n}\delta^-(\phi^{\rho_1},x_0,y,t_0)|y|^{-n-sp}\, dy+D\\
&\leq 2\int_{\R^n}\delta^-(\phi^{\rho_2},x_0,y,t_0)|y|^{-n-sp}\, dy+D<\infty,\nonumber
\end{align}
where we have used \eqref{deltaminus}. 

\par Since $\delta^+(\phi^\rho,x_0,y,t_0)\nearrow \delta^+(v,x_0,y,t_0)$, the monotone convergence theorem implies
$$
\int_{\R^n}\delta^+(\phi^\rho,x_0,y,t_0)|y|^{-n-sp}\, dy \to \int_{\R^n}\delta^+(v,x_0,y,t_0)|y|^{-n-sp}\, dy.
$$
By \eqref{eq:srineq}
\begin{equation}\label{eq:deltaplus}
2\int_{\R^n}\delta^+(v,x_0,y,t_0)|y|^{-n-sp}\, dy\leq 2\int_{\R^n}\delta^-(\phi^\rho,x_0,y,t_0)|y|^{-n-sp}\, dy+D<\infty,
\end{equation}
for any $0<\rho<r$. We conclude that $\delta^+(v,x_0,y,t_0)|y|^{-n-sp}$ is integrable. By the dominated convergence theorem, we can pass to the limit in the right hand side of \eqref{eq:deltaplus} and obtain
$$
2\int_{\R^n}\delta^+(v,x_0,y,t_0)|y|^{-n-sp}\, dy\leq 2\int_{\R^n}\delta^-(v,x_0,y,t_0)|y|^{-n-sp}\, dy+D<\infty.
$$
This is simply another way of writing
$$
2\int_{\R^n}\delta (v,x_0,y,t_0)|y|^{-n-sp}\, dy\leq D.
$$
Therefore $(-\lap_p)^s v\,(x_0,t_0)$ exists in the pointwise sense and $(-\lap_p)^s v\, (x_0,t_0)\leq D$, which concludes the proof as $D=C-|\phi_t(x_0,t_0) |^{p-2}\phi_t(x_0,t_0)$.
\end{proof}

\begin{prop}\label{prop:comparison} Assume that $u$ is a viscosity subsolution and that $v$ is a viscosity supersolution of
$$
|v_t|^{p-2}v_t+(-\Delta_p)^s v=0, \quad \text{in }\Omega\times (0,T).
$$
Suppose $u,v\in L^\infty(\R^n\times[0,T])$, $u\le v$ in $\R^n\setminus\Omega\times [0,T]$ and
$$
\limsup_{(x,t)\to (x_0,t_0)} u(x_0,t_0)\leq \liminf_{(x,t)\to (x_0,t_0)} v(x_0,t_0), \quad \text{for }(x_0,t_0)\in \partial \Omega \times (0,T)\cup \Omega\times \{0\}.
$$
Then $u\leq v$.
\end{prop}
\begin{proof}
We employ the usual trick of adding a term $\frac{\delta}{t-T}$; let $\tilde u =u+\frac{\delta}{t-T}$. Then $v$ is a supersolution, $\tilde u$ is a subsolution of
$$
|v_t|^{p-2}v_t+(-\Delta_p)^s v=-\frac{\delta}{(t-T)^2},
$$
$\tilde u< v$ in $\R^n\setminus\Omega\times [0,T]$ and 
\begin{equation}\label{OrderingUtild}
\limsup_{(x,t)\to (x_0,t_0)} \tilde u(x_0,t_0)< \liminf_{(x,t)\to (x_0,t_0)} v(x_0,t_0),
\end{equation}
for $(x_0,t_0)\in \partial \Omega \times (0,T)\cup \Omega\times \{0\}$. Moreover, $\tilde u(x,t)-v(x,t)\to -\infty$ as $t\to T$. It is now sufficient to prove that $\tilde u\leq v$ for any $\delta>0$ since we can then let $\delta\to 0$. We argue by contradiction and assume that 
$$
\sup_{\R^n\times [0,T]} (\tilde u-v)>0.
$$
\par Fix $\e>0$ and define
%By assumption, the supremum must be attained at a point $(x_0,t_0)\in \Omega\times (0,T)$. 
\begin{align*}
M_\e &:=\sup_{ \R^n\times[0,T]\times \R^n\times [0,T]} \left(\tilde u(x,t)-v(y,\tau)-\frac{|x-y|^2+|t-\tau|^2}{\e}\right).
%&= \left(\tilde u(x_\e,t_\e)-v(y_\e,\tau_\e)-\frac{|x_\e-y_\e|^2+|\tau_\e-t_\e|^2}{\e}\right)\\
\end{align*}
Note $M_\e\ge \sup_{\R^n\times [0,T]} (\tilde u-v)>0$ and select $x_\e,y_\e\in \R^n$ and $t_\e,\tau_\e\in [0,T]$ for which  
$$
M_\e<\tilde u(x_\e,t_\e)-v(y_\e,\tau_\e)-\frac{|x_\e-y_\e|^2+|\tau_\e-t_\e|^2}{\e}+\e.
$$
By Proposition 3.7 in \cite{User}, $(x_\e, t_\e)$ and $(y_\e,\tau_\e)$ each have subsequences converging to $(\hat x,
\hat t)\in \Omega\times(0,T)$ as $\e\rightarrow 0$ for which 
$$
\sup_{\R^n\times [0,T]} (\tilde u-v)=(\tilde u-v)(\hat x,
\hat t).
$$
As a result, there is $\e$ small enough such that $x_\e,y_\e\in \Omega$ and $t_\e,\tau_\e\in (0, T)$.  For this $\e$, it also follows that the maximum $M_\e$ is attained in $\Omega\times(0,T)\times\Omega\times (0,T)$. For convenience, 
we will again call this point $(x_\e, t_\e,y_\e,\tau_\e)$. 
%$$
%M_\e=\tilde u(x_\e,t_\e)-v(y_\e,\tau_\e)-\frac{|x_\e-y_\e|^2+|\tau_\e-t_\e|^2}{\e}+\e
%$$

\par Observe that the function
$$\frac{|x-y_\e|^2+|t-\tau_\e|^2}{\e}+\tilde u(x_\e,t_\e)-\frac{|x_\e-y_\e|^2+|t_\e-\tau_\e|^2}{\e}
$$
touches $\tilde u$ from above at $(x_\e,t_\e)$ and
$$\frac{|x_\e-y|^2+|t_\e-\tau|^2}{\e}-v(y_\e,\tau_\e)-\frac{|x_\e-y_\e|^2+|t_\e-\tau_\e|^2}{\e}
$$
touches $-v$ from above at $(y_\e,\tau_\e)$. From Proposition \ref{pw}, we can conclude that $(-\lap_p)^s \tilde u(x_\e,t_\e)$ and $(-\lap_p)^s v(y_\e,\tau_\e)$ exist pointwise and satisfy
$$
(-\lap_p)^s \tilde u\,(x_\e,t_\e)<(\lap_p)^s v\,(y_\e,\tau_\e).
$$
\par In addition, since the function 
$$
\tilde u(x,t)-v(y,\tau)-\frac{|x-y|^2+|t-\tau|^2}{\e}
$$
is larger at $(x_\e,y_\e,t_\e,\tau_\e)$ than at $(x_\e+y,y_\e+y,t_\e,\tau_\e)$ for any $y$, we obtain 
$$
\tilde u(x_\e,t_\e)-\tilde u(x_\e+y,t_\e)\geq v(y_\e,\tau_\e)-v(y_\e+y,\tau_\e).
$$
This implies
$$
(-\lap_p)^s \tilde u\, (x_\e,t_\e)\geq (-\lap_p)^s \, v(y_\e,\tau_\e), 
$$
which is a contradiction. Therefore, we must have $\tilde u \leq v$.
\end{proof}
Now we present a general result for nonlocal parabolic equations, inspired by previous work of Petri Juutinen in Theorem 1 of \cite{Jut01}. This fact will be important in the proof of H\"older regularity of solutions of \eqref{MainPDE}. 
\begin{prop}\label{prop:jut} Suppose that $v$ is a viscosity subsolution of
$$
|v_t|^{p-2}v_t+(-\lap_p)^s v\,\leq 0
$$
in $B_r(x_0)\times (t_0-r,t_0+r)$ and $\phi\in C^{2,1}_{x,t}({B_r(x_0)}\times (t_0-r,t_0+r))$.  If
\begin{equation}\label{eq:past}
\phi(x_0,t_0)=v(x_0,t_0), \quad \phi(x,t)\geq v(x,t) \text{ for $(x,t)\in B_r(x_0)\times (t_0-r,t_0]$},
\end{equation}
then
%that $\phi$ satisfies \eqref{eq:past}. Then
$$
|\phi_t(x_0,t_0) |^{p-2}\phi_t(x_0,t_0)+(-\lap_p)^s\phi^r\, (x_0,t_0)\leq 0.
$$
\end{prop}
\begin{proof} We argue by contradiction. If the assertion is not true then
$$
|\phi_t(x_0,t_0) |^{p-2}\phi_t(x_0,t_0)+(-\lap_p)^s\phi^r\, (x_0,t_0)\geq\e>0
$$
for some $\e$. Recall $\phi^r$ is defined in Definition \ref{paravisc}. By continuity, we have
$$
|\phi_t |^{p-2}\phi_t+(-\lap_p)^s\phi^r \geq\e/2>0
$$
in $B_\rho(x_0)\times (t_0-\rho,t_0)$ for $\rho$ small enough. 

\par Let $\eta:\R^{n+1}\to \R$ be a smooth function satisfying  
$$
\begin{cases}
0\leq \eta\leq 1,\\
\eta(x_0,t_0)=0,\\
\eta(x,t)>0 \text{ if $(x,t)\neq (x_0,t_0)$},\\
\eta(x,t)\ge1 \text{ if $(x,t)\not\in B_{\rho}(x_0)\times (t_0-\rho,t_0)$}.
\end{cases}
$$
Also define
$$ \phi_\delta(x,t)=\phi(x,t)+\delta\eta(x,t)-\delta,
$$
where $\delta>0$ is considered small. By continuity, 
$$
|(\phi_\delta)_t |^{p-2}(\phi_\delta)_t+(-\lap_p)^s(\phi_\delta)^r \geq\e/4>0
$$
in $B_{\rho}(x_0)\times (t_0-\rho,t_0)$, provided $\delta$ is small enough.

This means that $(\phi_\delta)^r$ is a supersolution in the pointwise classical sense in $B_{\rho}(x_0)\times (t_0-\rho,t_0)$, and in particular it means that $(\phi_\delta)^r$ is a viscosity supersolution in this region. Moreover, $(\phi_\delta)^r\ge\phi^r\geq v$ in the complement of 
$$
\R^n\setminus{B_{\rho}(x_0)}\times (t_0-\rho,t_0)\cup B_{\rho}(x_0)\times \{t_0-\rho\}.
$$
By Proposition \ref{prop:comparison}, $(\phi_\delta)^r\geq v$ in $\R^n\times [t_0-\rho,t_0]$. Furthermore, $(\phi_\delta)^r(x_0,t_0)\geq v(x_0,t_0)$ which is a contradiction since $(\phi_\delta)^r(x_0,t_0)=\phi(x_0,t_0)-\delta=v(x_0,t_0)-\delta$.
\end{proof}

\par Let us now return to our study of the implicit time scheme. We are now in position to construct barriers that assure that 
$\overline{v}$ and $\underline{v}$ satisfy the correct boundary and initial conditions.
\begin{lem}\label{lem:BC}
Assume that $-\Psi\leq g\leq \Psi$ where $\Psi$ is a non-negative ground state of $\laps$. Then $\overline v$ and  $\underline v$ satisfy the boundary condition in the classical sense, i.e., 
$$
\lim_{y\to x} \overline v(y,t)=\lim_{y\to x} \underline v(y,t)=0, \quad \text{for any $x\in \partial \Omega$ and any $t\geq 0$}.
$$
\end{lem}
\begin{proof} We observe that
$$
-\laps \Psi-\frac{|\Psi-g|^{p-2}(\Psi-g)}{\tau^{p-1}}=-\lambda_{s,p}\Psi^{p-1}-\frac{(\Psi-g)^{p-1}}{\tau^{p-1}}\leq 0.
$$
Hence $\Psi$ is a supersolution of \eqref{ViscScheme}. Since $\Psi=v^1=0$ in $\R^n\setminus  \Omega$, the comparison principle implies
$$
v^1\leq \Psi.
$$
We can argue similarly to obtain 
$$
|v^1|\leq \Psi.
$$
Iterating this method for each $v^k$ yields $|v^k|\leq \Psi$ for any $k=1,\dots, N$. By the definition of $v_N$ in \eqref{StepApprox},
\begin{equation}\label{veeNEstim}
|v_N|\leq \Psi.
\end{equation}
\par By inequality \eqref{veeNEstim}, the assertion would follow as long as $\Psi$ is continuous up to the boundary. 
To establish this continuity, we first note that $\Psi$ is globally bounded. This fact is due to Theorem 3.2 in \cite{FP14}, 
Theorem 3.3 in \cite{BLP14} or Theorem 3.1 together with Remark 3.2 in \cite{BP14}. Theorem 1.1 in \cite{IMS15}
can now be used to establish the desired continuity of $\Psi$.  
\end{proof}

\begin{lem}\label{lem:IC} Assume $g$ is continuous. Then $\overline v$ and  $\underline v$ satisfy the initial condition in the classical sense, i.e., 
$$
\lim_{t\to 0} \overline v(x,t)=\lim_{t\to 0}\underline v (x,t)=g(x), \quad \text{for any $x\in \Omega$}.
$$
\end{lem}
\begin{proof} Take $\eta$ to be a bounded, smooth and strictly increasing radial function such that $\eta(0)=0$. Let $d=\text{diam}\, \Omega$ and 
$$
\alpha^{p-1} = \sup_{x\in B_d}\Big|\laps \eta\, (x)\Big|.
$$
Clearly $\alpha$ is finite. Now we fix $x_0\in \Omega$. We first prove that given $\e>0$ there is $C=C(x_0,\e)$ such that 
$$
u(x)=g(x_0)+\e+C\left(\alpha \tau+\eta(x-x_0)\right)
$$
lies above $v^1$. 

\par As $g$ is continuous, for each $\e>0$ there is $\delta>0$ and $C>0$ so that
$$
|g(x)-g(x_0)|<\e \quad \text{if }|x-x_0|<\delta
$$
and
$$
\sup |g|\leq C \eta(x-x_0)  \quad \text{if } |x-x_0|\geq \delta.
$$
Upon choosing $C$ even larger, we may also assume that $u\geq 0$ in $\R^n\setminus  \Omega$. In addition
\begin{align*}
-\laps u-\frac{|u-g|^{p-2}(u-g)}{\tau^{p-1}}&\leq C^{p-1}\alpha^{p-1} -\frac{\left(g(x_0)+\e+C\eta(\cdot -x_0)-g+C\alpha \tau\right)^{p-1}}{\tau^{p-1}}\\
&\leq C^{p-1}\alpha^{p-1} -C^{p-1}\alpha^{p-1} =0, 
\end{align*}
since $g(x_0)+\e+C\eta(\cdot-x_0)-g(\cdot)\geq 0$ by construction. Now it follows from the comparison principle that
$$
v^1(x)\leq u(x)=g(x_0)+\e+C\left(\alpha \tau+\eta(x-x_0)\right).
$$
Arguing in the same fashion, we have
$$
v^1(x)\geq g(x_0)-\e-C\left(\alpha \tau+\eta(x-x_0)\right).
$$
\par Similarly we can obtain the bounds
$$
g(x_0)-\e-C\left(k\alpha \tau+\eta(x-x_0)\right)\leq v^k(x)\leq g(x_0)+\e+C\left(k\alpha \tau+\eta(x-x_0)\right)
$$
for each $k=1,\dots, N$.  Using the definition \eqref{StepApprox} of $v_N$, we also have
$$
v_N(x,t)\leq v^{k}(x)\leq g(x_0)+\e+C\left(\alpha k\tau+\eta(x-x_0)\right)\leq g(x_0)+\e+C\left(\alpha t+\alpha\frac{T}{N}+\eta(x-x_0)\right)
$$
for $t\in ((k-1)\tau,k\tau)$ as $\tau =T/N$. A similar estimate from below holds, as well. In total, 
$$
g(x_0)-\e-C\left(\alpha t+\alpha\frac{T}{N}+\eta(x-x_0)\right)\leq v_N(x,t)\leq g(x_0)+\e+C\left(\alpha t+\alpha\frac{T}{N}+\eta(x-x_0)\right).
$$

\par Passing to the liminf and limsup in the above inequalities, we find
$$
g(x_0)-\e-C\left(\alpha t+\eta(x-x_0)\right)\leq \underline v(x,t)\leq \overline v(x,t)\leq g(x_0)+\e+C\left(\alpha t+\eta(x-x_0)\right).
$$
And after letting $x= x_0$ and $t\to 0$ 
$$
g(x_0)-\e\leq \liminf_{t\to 0}\underline v(x_0,t)\leq \limsup_{t\to 0}\overline v(x_0,t)\leq g(x_0)+\e.
$$
Since both $\e$ and $x_0\in \Omega$ are arbitrary, the desired result follows.
\end{proof}

\begin{proof}[Proof of Proposition \ref{ViscSolnResult}] It is enough to show that $\overline{v}$ is a viscosity subsolution of \eqref{MainPDE}. The same argument (applied to $-\overline{v}$) yields that $\underline{v}$ is a supersolution. Combining Lemma \ref{lem:BC}, Lemma \ref{lem:IC}, and Proposition \ref{prop:comparison}, would then imply $\overline{v}\le \underline{v}$. Hence, $v:=\overline{v}=\underline{v}$ is continuous and $v_N$ converges to $v$ locally uniformly. The claim would then follow as $v_N$ has a subsequence converging to a weak solution of \eqref{MainPDE} in $C([0,T], L^p(\Omega))$; see Remark \ref{rem:diffeq}.

We now prove that $\overline{v}$ is a viscosity subsolution of \eqref{MainPDE}. Assume that $\phi\in C^{2,1}_{x,t}({B_r(x_0)}\times (t_0-r,t_0+r))$ and $\overline{v}-\phi$ has a strict maximum in $B_r(x_0)\times (t_0-r,t_0+r)$ at $(x_0,t_0)\in \Omega\times(0,T)$. By Lemma \ref{ApproxPoints}, there are points 
$(x_j,t_j)$ converging to $(x_0,t_0)$ and $N_j\in \N$ tending to $+\infty$, as $j\rightarrow \infty$, such that $v_{N_j}-\phi_{N_j}$ has a maximum in $B_r(x_0)\times (t_0-r,t_0+r)$ at $(x_j,t_j)$. Observe that for each $j\in \N$, 
$t_j\in (\tau_{k_j-1}, \tau_{k_j}]$ for some $k_j\in\{0,1,\dots, N_j\}$. Hence, by the definition of $v_{N_j}$ and $\phi_{N_j}$, 
$$
\Omega\times \{0,1,\dots,N_j\}\ni (x,k)\mapsto v^k(x)- \phi(x,\tau_{k})
$$
has a local maximum in $B_r(x_0)\times \{0,1,\dots,N_j\}$ at $(x,k)=(x_j, k_j)$.  By Lemma \ref{LemdiscreteVisc}, 
$$
{\cal J}_p\left(\frac{ \phi(x_j,\tau_{k_j})-  \phi(x_j,\tau_{k_j-1})}{T/N_j}\right)+(-\Delta_p)^s \phi^r(x_j,\tau_{k_j})\le 0.
$$
As $\tau_{k_j-1}=\tau_{k_j}-T/N_j$ and $|t_j-\tau_{k_j}|\le T/N_j$ for $j\in \N$, we can send $j\rightarrow \infty$ above by appealing to the smoothness of $\phi$ and arrive at 
$$
{\cal J}_p(\phi_t(x_0,t_0))+(-\Delta_p)^s\phi^r(x_0,t_0)\leq 0.
$$
It follows that $\overline{v}$ is a viscosity subsolution.
\end{proof}

\section{H\"older estimates for viscosity solutions}\label{sec:holder}
In this section we prove Theorem \ref{Holderthm}. The proof of this regularity result is based on Lemma \ref{lem:key} below. We start by noting an elementary inequality that will come in handy:
\begin{equation}
\label{eq:pest}
|a+b|^{p-2}(a+b)\leq 2^{p-2}(|a|^{p-2}a+|b|^{p-2}b),\quad  a+b\geq 0,\quad  p\geq 2.
\end{equation}
\begin{lem}\label{lem:key} Fix $\delta>0$. Suppose $v$ is continuous in $Q_1^-$ and satisfies (in the viscosity sense)
\begin{align*}
|v_t|^{p-2}v_t+(-\lap_p)^s v\leq 0\text{ in }Q_1^-,\\
v\leq 1\text{ in }Q_1^-,\\
v(x,t)\leq 2|2x|^\eta-1\text{ in }\R^n\setminus B_1\times (-1,0),\\
%|\{B_1\times [-1,-\frac{1}{2^\frac{sp}{p-1}}]\}\cap \{v\leq 0\}|>\delta,
\Big|\left\{B_1\times \left [-1,-\frac{1}{2^\frac{sp}{p-1}}\right ]\right\}\bigcap \{v\leq 0\}\Big|>\delta.
\end{align*}
Then for $\eta$ small enough, $v\leq 1-\theta<1$ in $Q_{1/2}^-$, where $\theta =\theta(\delta,p,s)>0$.
\end{lem}
Recall that the parabolic cylinders $Q_1^-$ and $Q_\frac12^-$ have been defined on page \ref{p:def}. Before proving this lemma, we will first need to gain control of a certain function.
\begin{lem}\label{lem:mblem} Fix $\delta>0$, let $\e>0$ and assume the following
\begin{align*}
m(t)=e^{-c_1t}\int_{-1}^t c_0e^{c_1s}|G(s)|ds,\\
%m(-1)=0,\\
%m'(t)	=c_0|G(t)|-c_1m(t),\\
G(t)=\{x\in B_1: v(x,t)\leq 0\},\\
b(x,t)=1+\e-m(t)\rho(x),\\
0\leq \rho\leq 1, \quad \rho = 1 \text{ in $B_\frac12$},\quad  \rho\in C_0^\infty(B_\frac34), \\
\Big|\left\{B_1\times \left [-1,-\frac{1}{2^\frac{sp}{p-1}}\right ]\right\}\bigcap \{v\leq 0\}\Big|>\delta.
\end{align*}
If $c_0|B_1|\leq c_1/2$ then
$$
b(x,t)\geq \frac12
$$
and for $0\geq t\geq -1/2^\frac{sp}{p-1}$,
$$
m(t)\geq c_0 e^{-c_1} \delta.
$$
\end{lem}
\begin{rem} Note that $m$ solves the equation
$$
m'(t)	=c_0|G(t)|-c_1m(t)
$$
for a.e. $t\in [-1,0]$.
\end{rem}
\begin{proof} As $|G(t)|\leq |B_1|$, it follows that $m(t)\leq c_0/c_1 |B_1|$. And since $c_0|B_1|\leq c_1/2$, 
$$
b(x,t)\geq 1+\e-\frac{c_0}{c_1}|B_1|\rho(x)\geq 1+\e-\frac{c_0}{c_1}|B_1|\geq \frac12.
$$
 Moreover, 
$$
m(t)=e^{-c_1t}\int_{-1}^t c_0e^{c_1s}|G(s)|ds\geq  e^{c_1(-1-t)}\int_{-1}^{t} c_0|G(s)| ds.
$$
From our hypotheses, 
$$
\int_{-1}^{-\frac{1}{2^\frac{sp}{p-1}}} |G(s)| ds\geq \delta.
$$
Therefore, 
$$
m(t)\geq c_0 e^{-c_1}\delta,
$$
for $0\geq t\geq -1/2^\frac{sp}{p-1}$.
\end{proof}
\begin{proof}[~Proof of Lemma \ref{lem:key}]
Assume the hypotheses of Lemma \ref{lem:mblem}. Choose $c_0$, $c_1$ and $\e$ so that
$$
c_0^{p-1} < 1/(2^{2p-4+n+sp}|B_1|^{p-2}), \quad c_0|B_1|\leq c_1/2, \quad 2^{2-p}(c_1)^{p-1}>2\sup_{x_0\in B_\frac34}(-\lap_p)^s\left(\frac{\rho}{\rho(x_0)}\right)\,(x_0), %c_0|B_1|\leq \min (1,c_1/4)
$$
and
$$
2\e<e^{-c_1}c_0\delta.
$$
Note that the quantity 
$$
2\sup_{x_0\in B_\frac34}(-\lap_p)^s\left(\frac{\rho}{\rho(x_0)}\right)\,(x_0)
$$
is finite, since the only way it could be infinite, is if there is maximizing sequence of points $x_j$ where $\rho(x_j)\to 0$. But then 
$$
(-\lap_p)^s\left(\frac{\rho}{\rho(x_j)}\right)\,(x_j)
$$
would be negative for $j$ large enough.

\par We claim that $v\leq b$ in $Q_1^-$. Let us describe how the lemma follows once this claim is proved. By the lower bound on $m$ in Lemma \ref{lem:mblem} we have
$$
b(x,t)\leq 1+\e-e^{-c_1}c_0\delta
$$
for $0\geq t\geq -1/2^\frac{sp}{p-1}$.
Since $2\e<e^{-c_1}c_0\delta $,
$$
b(x,t)\leq 1-\frac{e^{-c_1}c_0\delta}{2}.
$$
Therefore, 
$$
v\leq b\leq 1-\theta
$$
in $Q_\frac12^-$ as long as we choose 
$$
\theta = \frac{e^{-c_1}c_0\delta}{2}.
$$

\par Let us now prove that $v\leq b$ in $Q_1^-$. We argue by contradiction. Assume that, starting from $t=-1$, the first time $v$ touches $b$ at some point in $Q_1^-$ is at the point $(x_0,t_0)$. Since $\rho =0$ outside $B_\frac34$ and $v\leq 1$ in $Q_1^-$, we know that $x_0\in B_\frac34$. In addition, since $m(-1)=0$, we know $t_0>-1$. It is not difficult to see $b$ touches $v$ from above at $(x_0,t_0)$ in the sense of \eqref{eq:past}. In order to simplify the presentation, we first assume that $b$ is $C^1$ at $(x_0,t_0)$ and explain in the last paragraph of this proof how to relax this assumption. 

\par By Proposition \ref{pw}, $(-\lap_p)^s  v (x_0,t_0)$ is well defined and
\begin{equation}\label{eq:bveq}
|b_t(x_0,t_0)|^{p-2}b_t(x_0,t_0) +(- \lap_p)^s v(x_0,t_0)\leq 0, 
\end{equation}
%at $(x_0,t_0)$. 
Note that $b_t(x_0,t_0)=-m'(t_0)\rho(x_0)$.  We will now estimate $(-\lap_p)^s (b-v)\, (x_0,t_0)$ from above and from below and arrive at a contradiction. This part of the proof will be divided into four steps. Along the way, we will use the notation
$$
L_D w\, (x,t):=2\, \pv \int_{y\in D} \frac{\J(w(y,t)-w(x,t))}{|x-y|^{n+ps}} d y, 
$$
for a measurable function $w$ and an open or closed set $D\subset \R^n$. Notice that
$$
(-\lap_p)^s w = -L_{\R^n} w.
$$

\noindent {\bf Step 1: Estimate $L_{B_1}$}\\
Since $b(\cdot,t_0)\geq v(\cdot,t_0)$ in $B_1$, \eqref{eq:pest} implies
\begin{align*}
L_{B_1} (b-v)\, (x_0,t_0)&=2\,\pv\int_{B_1} \frac{\J((b-v)(y,t_0)-(b-v)(x_0,t_0))}{|x_0-y|^{n+ps}} dy\\
&\leq 2^{p-1}\pv\int_{B_1}\frac{\J(b(y,t_0)-(b(x_0,t_0))-\J(v(y,t_0)-v(x_0,t_0))}{|x_0-y|^{n+ps}} dy\\
&=2^{p-2} \left(L_{B_1} b\, (x_0,t_0) - L_{B_1} v\, (x_0,t_0)\right).
\end{align*}
In addition, since $v(x_0,t_0)=b(x_0,t_0)$
\begin{align*}
L_{B_1} (b-v)\, (x_0,t_0)&=2\,\pv\int_{B_1} \frac{\J((b-v)(y,t_0))}{|x_0-y|^{n+ps}} dy\\
&\geq 2\int_{G(t_0)} \frac{|b(y,t_0)|^{p-2}b(y,t_0)}{|x_0-y|^{n+ps}} dy\\
&\geq 2\left(\frac12\right)^{n+sp}\inf_{y\in B_1}|b(y,t_0)|^{p-1} |G(t_0)|\\
&\geq \left(\frac12\right)^{p-2+n+sp}| G(t_0)|,
\end{align*}
from Lemma \ref{lem:mblem}.

\noindent {\bf Step 2: Estimate $L_{\R^n\setminus B_1}$}\\
By our hypotheses,
$$
v(y,t_0)\leq 2|2y|^\eta-1,\quad b=1+ \e>1
$$
whenever $|y|>1$.
Hence, $b(y,t_0)-v(y,t_0)\geq 2(1-|2y|^\eta)$ so that
\begin{equation}
\label{eq:alphaest1}
b(y,t_0)-v(y,t_0)\leq b(y,t_0)-v(y,t_0)+2(|2 y|^\eta -1)
\end{equation}
and
\begin{equation}\label{eq:alphaest2}
 b(y,t_0)-v(y,t_0)+2(|2 y|^\eta -1) \geq 0.
\end{equation}
\par By \eqref{eq:pest}, \eqref{eq:alphaest1} and \eqref{eq:alphaest2}
\begin{align*}
L_{\R^n \setminus B_1} (b-v)\,(x_0,t_0) &\leq 2\int_{\R^n\setminus B_1}\frac{\J(b(y,t_0)-v(y,t_0)+2(|2y|^\eta-1) -(b(x_0,t_0)-v(x_0,t_0)))}{|x_0-y|^{n+ps}} dy\\
&\leq 2^{p-2}\left(-L_{\R^n \setminus B_1} \, v(x_0,t_0)+2\int_{\R^n\setminus B_1} \frac{\J(b(y,t_0)+2(|2y|^\eta-1)-b(x_0,t_0))}{|x_0-y|^{n+ps}} dy\right).
\end{align*}
Using \eqref{eq:alphaest2}, we obtain the estimate from below
\begin{align*}
L_{\R^n \setminus B_1} (b-v)\,(x_0,t_0) &= 2\int_{\R^n\setminus B_1}\frac{\J(b(y,t_0)-v(y,t_0))}{|x_0-y|^{n+ps}} dy\\
&\geq -2\int_{\R^n\setminus B_1}\left(2(|2y|^\eta-1)\right)^{p-1} \frac{dy}{|x_0-y|^{n+sp}}\\
&:=-c_\eta.
\end{align*}
We note that $\lim_{\eta\rightarrow 0^+}c_\eta=0$ by an application of the dominated convergence theorem.\\

\noindent {\bf Step 3: Use the equation}\\
The two steps above together imply
\begin{align}\label{eq:ineq1}
\left(\frac12\right)^{p-2+n+sp} |G(t_0)|-c_\eta &\leq -(-\lap_p)^s(b-v)\, (x_0,t_0)\nonumber \\
&\leq 2^{p-2} (-\lap_p)^s  v\, (x_0,t_0)+2^{p-2} L_{B_1} b(x_0,t_0)\\
&+2^{p-1}\int_{\R^n\setminus B_1}\frac{\J(b(y,t_0)+2(|2y|^\eta-1)-b(x_0,t_0)}{|x_0-y|^{n+ps}} dy\nonumber .
\end{align}

\par From inequality \eqref{eq:bveq}, it follows that
\begin{align}\label{eq:ineq2}
(-\lap_p)^s v\, (x_0,t_0)&\leq -|b_t|^{p-2}b_t (x_0,t_0)\nonumber \\
&= |m'(t_0)\rho(x_0)|^{p-2} m'(t_0)\rho(x_0)\\\nonumber 
&= |\rho(x_0)c_0|G(t_0)|-\rho(x_0)c_1 m(t_0)|^{p-2}(\rho(x_0)c_0|G(t_0)|-\rho(x_0)c_1 m(t_0))\\
&\leq |c_0 |G(t_0)|-\rho(x_0)c_1 m(t_0)|^{p-2}(c_0|G(t_0)|-\rho(x_0)c_1 m(t_0)).\nonumber 
\end{align}
Using \eqref{eq:pest}, with $a=c_1\rho(x_0)m(t_0)-c_0|G(t_0)|$ and $b=c_0|G(t_0)|$, we then obtain 
\begin{align*}
&(c_1 \rho(x_0)m(t_0))^{p-1}\leq \\
&2^{p-2}| \rho(x_0)c_1 m(t_0)-c_0|G(t_0)||^{p-2}(\rho(x_0)c_1 m(t_0)-c_0|G(t_0)|)+2^{p-2}(c_0|G(t_0)|)^{p-1}.
\end{align*}
After rearranging
\begin{align}\label{eq:ineq3}
&2^{p-2}|c_0 |G(t_0)|-\rho(x_0)c_1 m(t_0)|^{p-2}(c_0|G(t_0)|-\rho(x_0)c_1 m(t_0))\leq \\
&2^{p-2}(c_0|G(t_0)|)^{p-1}-(c_1 \rho(x_0)m(t_0))^{p-1} .\nonumber 
\end{align}

Combining \eqref{eq:ineq1}, \eqref{eq:ineq2} and \eqref{eq:ineq3} yields
\begin{align*}
&(c_1 \rho(x_0)m(t_0))^{p-1}-2^{p-2}(c_0|G(t_0)|)^{p-1}+\left(\frac12\right)^{p-2+n+sp} |G(t_0)|-c_\eta \\
&\leq 2^{p-2}L_{B_1} b(x_0,t_0)+2^{p-1}\int_{\R^n\setminus B_1}\frac{\J(b(y,t_0)+2(|2y|^\eta-1)-b(x_0,t_0))}{|x_0-y|^{n+ps}} dy.
\end{align*}

\par Since we assumed at the outset that $c_0^{p-1} < 1/(2^{2p-4+n+sp}|G(t_0)|^{p-2})$, we have by the definition of $b$ and $L_{B_1}$
\begin{align}\label{eq:lastineq}
&2^{2-p}(c_1\rho(x_0) m(t_0))^{p-1}\\
&\leq 2^{2-p}c_\eta+2\,\pv\int_{B_1} \frac{\J(m(t_0)(\rho(x_0)-\rho(y))}{|x_0-y|^{n+ps}} dy +2\int_{\R^n\setminus B_1}\frac{\J(m(t_0)\rho(x_0)+2(|2y|^\eta-1))}{|x_0-y|^{n+ps}} dy.\nonumber 
\end{align}
Here we also used that $\rho(y)=0$ whenever $y\not\in B_1$.

\noindent {\bf Step 4: Arrive at a contradiction}\\
It follows from the proof of Lemma \ref{lem:mblem} that $m$ is uniformly bounded with respect to $\eta$. Consequently, the second integral on the right hand side of \eqref{eq:lastineq} is uniformly bounded for all small $\eta$. We can again apply the dominated convergence theorem to show that the right hand side of \eqref{eq:lastineq} converges to the quantity
$$
-(-\lap_p)^s b\, (x_0,t_0) =  (m(t_0))^{p-1} (-\lap_p)^s\rho\, (x_0),
$$
as $\eta\to 0$. As $m$ is bounded from below by 1/2 (by Lemma \ref{lem:mblem}), there is $\gamma_\eta\searrow0$ as $\eta\to 0$ such that 
\begin{equation}\label{eq:etaeq}
2^{2-p}(c_1\rho(x_0))^{p-1}\leq \gamma_\eta +(-\lap_p)^s\rho\, (x_0).
\end{equation}
In general, $x_0$ will depend on $\eta$. Let us now consider two cases depending on the size of $(-\lap_p)^s \rho\,(x_0)$ for $\eta$ small.

\par For the first case, we suppose $\limsup_{\eta\rightarrow 0^+}(-\lap_p)^s \rho\,(x_0)\le 0$. Then \eqref{eq:etaeq} forces $\lim_{\eta\rightarrow 0^+}\rho(x_0)=0$ as $\eta\to 0$. It would then follow that $(-\lap_p)^s \rho\,(x_0)< -\gamma_\eta$ for all small $\eta>0$. Together with \eqref{eq:etaeq}, this would in turn would force $\rho(x_0)<0$ for $\eta$ small enough, which is a contradiction. 

\par Alternatively if $\limsup_{\eta\rightarrow 0^+}(-\lap_p)^s \rho\,(x_0)>0$, then for some sequence of $\eta\to 0$,
$(-\lap_p)^s \rho\, (x_0)\ge \gamma_{\eta}$. By \eqref{eq:etaeq} 
$$
2^{2-p}(c_1\rho(x_0))^{p-1}\leq 2(-\lap_p)^s\rho\, (x_0)
$$
along this sequence. Also note that $(-\lap_p)^s \rho\, (x_0)>0$ implies that $x_0\in B_\frac34$.
After dividing by $(\rho(x_0))^{p-1}$, we have
$$
2^{2-p}(c_1)^{p-1}\leq 2(-\lap_p)^s\left(\frac{\rho}{\rho(x_0)}\right)\, (x_0)\leq 2\sup_{x_0\in B_\frac34}(-\lap_p)^s\left(\frac{\rho}{\rho(x_0)}\right)\,(x_0).
$$
However, by our hypotheses on $c_1$
$$
2^{2-p}(c_1)^{p-1}>2\sup_{x_0\in B_\frac34}(-\lap_p)^s\left(\frac{\rho}{\rho(x_0)}\right)\,(x_0),
$$
which is a contradiction.

\noindent {\bf Step 5: Relax the $C^1$ assumption on $b$}\\
As mentioned above, $m$ is not necessarily $C^1$ since $|G(t)|$ is not necessarily continuous. We have chosen to ignore this fact in the reasoning above, in order to make the proof more accessible. This issue can be handled as follows. 

\par First, set 
$$
\chi_k(x,t): = \int^0_{-1}\phi_k(t-s)\chi_{\{v\leq 0\}}(x,s)ds, 
$$
for $x\in \R^n$ and $t\in \R$, where $\phi_k$ is a standard mollifier. Also define
$$
g_k(t)=\int_{B_1}\chi_k(x,t) dx.
$$
Observe $g_k(t)\to |G(t)|$ a.e. and in $L^1(\R)$ as $k\rightarrow\infty$. 

\par Now set 
$$
m_k(t):=e^{-c_1t}\int_{-1}^t c_0e^{c_1s}g_k(s)ds
$$
for $t\in [-1,0]$ and 
$$
b_k(x,t):=1+\e-m_k(t)\rho(x)
$$
for $(x,t)\in Q^-_1$.  It is evident that $m_k\to m$ and $b_k\to b$ uniformly as $k\to\I$.

\par Recall that  $b-v\ge\e$ and $b_k-v\ge\e$ on $\partial B_1\times[-1,0]\cup B_1\times\{-1\}$. These facts combined with the above uniform convergence implies that $v$ touches $b_k$ from below at some $(x_k,t_k)\to (x_0,t_0)$, where $v$ touches $b$ from below at $(x_0,t_0)$.  Without loss of 
generality, we may assume that $t_k<0$ for all $k\in \N$ large enough.  Moreover, as in Step 1 
we find
\begin{eqnarray}
\label{eq:chik}
2\,\pv\int_{B_1} \frac{\J((b_k-v)(y,t_k))}{|x_k-y|^{n+ps}} dy \geq 
 \displaystyle\left(\frac12\right)^{p-2+n+sp}\inf_{y\in B_1}\left(\frac{|b_k(y,t_k)|}{1/2}\right)^{p-1}  g_k(t_k)
 \nonumber\\-2\displaystyle\int_{B_1}(\chi_k-\chi_{\{v\leq 0\}})(y,t_k) \frac{|b_k(y,t_k)|^{p-2}b_k(y,t_k)}{|x_k-y|^{n+ps}} dy.
\end{eqnarray}
Notice that as $k\to \infty$
$$
\inf_{y\in B_1}\left(\frac{|b_k(y,t_k)|}{1/2}\right)^{p-1} \to \inf_{y\in B_1}\left(\frac{|b(y,t_0)|}{1/2}\right)^{p-1}\ge 1.
$$
Let us now argue that the second term on the right hand side of \eqref{eq:chik} goes to zero as $k\to \infty$.

\par By Lemma \ref{lem:mblem}, $b>0$ and so $b_k>0$ for all $k$ large enough. Hence, $v(x_k,t_k)=b(x_k,t_k)>0$. Since $v$ is continuous, $v>0$ in a neighborhood of $(x_k,t_k)$ for $k$ large. This means that $\chi_k=\chi_{\{v\leq 0\}}=0$ in $B_\tau(x_k)\times \{t_k\}$ if $\tau$ is small enough and $k$ large enough. Hence, 
$$
\int_{B_1}(\chi_k-\chi_{\{v\leq 0\}})(y,t_k) \frac{|b_k(y,t_k)|^{p-2}b_k(y,t_k)}{|x_k-y|^{n+ps}}dy=\int_{B_1\setminus B_\tau(x_k)}(\chi_k-\chi_{\{v\leq 0\}})(y,t_k) \frac{|b_k(y,t_k)|^{p-2}b_k(y,t_k)}{|x_k-y|^{n+ps}}dy.
$$
As a result, the integrand is uniformly bounded and converges to zero almost everywhere. By Lebesgue's dominated convergence theorem, we can conclude 
$$
c_k:=2\int_{B_1}(\chi_k-\chi_{\{v\leq 0\}})(y,t_k) \frac{|b_k(y,t_k)|^{p-2}b_k(y,t_k)}{|x_k-y|^{n+ps}} dy\to 0
$$
as $k\to\infty$.

\par  Steps 2 and 3 go through with minor modifications, so that we can obtain the following analog of \eqref{eq:lastineq}
\begin{align*}%\label{eq:lastineqwithn}\nonumber 
&2^{2-p}(c_1\rho(x_k) m_k(t_k))^{p-1}\leq 2^{2-p}(c_\eta+c_k)\\&+2\pv\int_{B_1} \frac{\J(m_k(t_k)(\rho(x_k)-\rho(y))}{|x_k-y|^{n+ps}} dy +2\int_{\R^n\setminus B_1}\frac{\J(m_k(t_k)\rho(x_k)+2(|2y|^\eta-1))}{|x_k-y|^{n+ps}} dy,\nonumber 
\end{align*}
for all $k$ sufficiently large. We can then send $k\to\infty$ and recover \eqref{eq:lastineq}. At this point, we can repeat Step 4 to complete this proof.
\end{proof}

We are now in a position to verify Theorem \ref{Holderthm} and prove that solutions of equation \eqref{MainPDE} are H\"older continuous.
\begin{proof}[~Proof of Theorem \ref{Holderthm}] Upon rescaling $v$ by the factor
$$
\frac{1}{2\|v\|_{L^\infty(\R^n\times [-2,0))}},
$$
we may assume that $v$ satisfies
$$
|v_t|^{p-2}v_t+(-\lap_p)^s v =0\text{ in }Q_2^-,\quad \osc_{\R^n\times [-2,0)}v \leq 1.
$$
We will now show that
$$
\displaystyle \osc_{{Q^-_{2^{-j}}(x_0,t_0)}} v\leq 2^{-j\alpha},\quad j=0,1,\ldots,  \text{ for any }(x_0,t_0)\in Q_1^-.
$$
Here $\alpha$ is chosen so that 
\begin{equation}\label{alphachoice}
\frac{2-\theta}{2}\leq 2^{-\alpha} \text{ and } \alpha\leq \eta, 
\end{equation}
%with $\delta = \delta(p,s):=(1-1/2^\frac{ps}{p-1})|B_1|/2$ and
where $\theta=\theta(\delta(p,s),p,s)$ and $\eta$ are from Lemma \ref{lem:key} with  $\delta(p,s):=(1-1/2^\frac{ps}{p-1})|B_1|/2$. This will imply the desired result with $C=2^{\alpha}$.

\par To this end, we will find constants $a_j$ and $b_j$ so that
\begin{equation}\label{eq:akbk}
b_j\leq v\leq a_j\text{ in  }Q^ -_{2^{-j}}(x_0,t_0),\quad |a_j-b_j|\leq 2^{-j\alpha}
\end{equation}
for $j\in \Z$. We construct these constants by induction on $j$. For $j\leq 0$, \eqref{eq:akbk} holds true with $b_j=\inf_{\R^n\times [-2,0)} v $ and $a_j=b_j+1$.  

\par Now assume \eqref{eq:akbk} holds for all $j\leq k$. We need to construct $a_{k+1}$ and $b_{k+1}$. Put $m_k=(a_k+b_k)/2$. Then
$$
|v-m_k|\leq 2^{-k\alpha-1}\text{ in  $Q^-_{2^{-k}}(x_0,t_0)$.}
$$
Let 
$$
w(x,t)=2^{\alpha k+1}(v(2^{-k}x+x_0,2^{-k\gamma}t+t_0)-m_k),\quad  \gamma = \frac{sp}{p-1}.
$$
Then 
$$
|w_t|^{p-2}w_t+(-\lap_p)^s w= 0 \text{ in }Q_1^-
$$
and
$$
 |w|\leq 1 \text{ in }Q_1^-.
$$
It also follows for $|y|>1$, such that $2^\ell\leq |y|\leq 2^{\ell+1}$, and $t\geq -2^{\gamma(\ell+1)}$ that 
\begin{align*}
w(y,t)= 2^{\alpha k+1}(v(2^{-k}y+x_0,2^{-k\gamma}t+t_0)-m_k)&\leq 2^{\alpha k+1}(a_{k-\ell-1}-m_k)\\
&\leq 2^{\alpha k+1}(a_{k-\ell-1}-b_{k-\ell-1}+b_{k}-m_k)\\
&\leq 2^{\alpha k+1}(2^{-\alpha(k-\ell-1)}-\frac12 2^{-k\alpha})\\
&\leq 2^{1+\alpha(\ell+1)}-1\leq 2|2y|^\alpha-1\\
&\leq 2|2y|^\eta-1.
\end{align*}
Here we used that \eqref{eq:akbk} holds for $j\leq k$. 

\par Suppose now that 
$$|\{(x,t): w(x,t)\leq 0\}\cap \{B_1\times [-1,-1/2]\}|\geq (1-1/2^\frac{ps}{p-1})|B_1|/2=\delta(p,s).$$
If not we would apply the same procedure to $-w$. Then $w$ satisfies all the assumptions of Lemma \ref{lem:key} with $\delta=\delta(p,s)$, and so 
$$
w\leq 1-\theta\text{ in }Q_\frac12^-. 
$$
Scaling back to $v$ yields
\begin{align*}
v(x,t)&\leq 2^{-1-\alpha k}(1-\theta)+m_k\leq 2^{-1-k\alpha}(1-\theta)+\frac{a_k+b_k}{2}\\
&\leq b_k+2^{-1-\alpha k}(1-\theta)+2^{-1-\alpha k}\\
&\leq b_k+2^{-\alpha (k+1)}
\end{align*}
for $(x,t)\in Q_{2^{-k-1}}(x_0,t_0)$, by \eqref{alphachoice}. Hence, if we let $b_{k+1}=b_k$ and $a_{k+1}=b_k+2^{-\alpha(k+1)}$ we obtain \eqref{eq:akbk} for the step $j=k+1$ and the induction is complete.
\end{proof}

\section{Large time limit}\label{sec:large}%\label{LargeTlim}\
% Formal L2 calculations in here
% Prove convergence theorem 
In this section, we prove Theorem \ref{LargeTthm} and Theorem \ref{LargeTthmvisc}. The main tools are the monotonicity of the Rayleigh quotient and the $W^{s,p}$-seminorm (Proposition \ref{RayleighComp} and Corollary \ref{RayleighGoDown}), the compactness of weak solutions (Theorem \ref{CompactnessLem}) and the H\"older estimates (Theorem \ref{Holderthm}). In order to control the sign of the limiting ground state we also need the following lemma.

\begin{lem}\label{lem:sign} Assume that $v$ is a weak solution of \eqref{pParabolic}. For any positive ground state $w$ for $\laps$ and any constant $C>0$,  there is a $\delta = \delta(w,C)>0$ such that if
\begin{enumerate}
\item $\displaystyle[v(\cdot,0)]^p_{W^{s,p}(\R^n)}\geq \lambda_{s,p}\int_\Omega |w|^p dx$,

\item $\displaystyle\int_\Omega |v(x,0)|^pdx\leq C$,

\item $\displaystyle \frac{[v(\cdot,0)]_{W^{s,p}(\R^n)}^p}{\displaystyle\int_\Omega |v(x,0)|^pdx}\leq \lambda_{s,p}+\delta$,

\item $\displaystyle\int_\Omega |v^+(x,0)|^pdx\geq\frac{1}{2}\int_\Omega |w|^pdx$,
\end{enumerate}
then
\begin{equation}\label{LemmaConclusion}
\int_\Omega |e^{\mu_{s,p} t} v^+(x,t)|^pdx\geq \frac{1}{2}\int_\Omega |w|^pdx,
\end{equation}
for $t\in [0,1]$.
\end{lem}
\begin{proof} We argue towards a contradiction. If the result fails, then there are $w$ and $C$ such that for every $\delta>0$, there is a weak solution $v$ that satisfies $(1)-(4)$ while 
\eqref{LemmaConclusion} fails. Therefore, associated to $\delta_j:=1/j$ $(j\in \N)$, there is a weak solution $v_j$ that satisfies 
\begin{enumerate}
\item $\displaystyle [v_j(\cdot,0)]^p_{W^{s,p}(\R^n)}\geq \lambda_{s,p}\int_\Omega |w|^p dx$,
\item $\displaystyle \int_\Omega |v_j(x,0)|^pdx\leq C$, 
\item $\displaystyle \frac{[v_j(\cdot,0)]_{W^{s,p}(\R^n)}^p}{\displaystyle\int_\Omega |v_j(x,0)|^pdx}\leq \lambda_{s,p}+\frac{1}{j} $,
\item $\displaystyle \int_\Omega |v_j^+(x,0)|^pdx\geq\frac{1}{2}\int_\Omega |w|^pdx$,
\end{enumerate}
while 
\begin{equation}\label{JustGonnaSendJ}
\int_\Omega |e^{\mu_{s,p} t_j} v_j^+(x,t_j)|^pdx< \frac{1}{2}\int_\Omega |w|^pdx 
\end{equation}
for some $t_j\in [0,1]$.  

\par Consequently, the sequence of initial conditions $(v_j(\cdot,0))_{j\in \N}$ is bounded in $W^{s,p}_0(\Omega)$ and has a subsequence (not relabeled) that converges to a positive ground state $g$ of $\laps$ in $W^{s,p}(\R^n)$. By Theorem \ref{CompactnessLem}, it also follows that (a subsequence of) the sequence of weak solutions $(v_j)_{j\in \N}$ converges to a weak solution $\tilde w$ in $C([0,2], L^p(\Omega))\cap L^p([0,2]; W^{s,p}(\R^n))$ with $\tilde w(\cdot, 0)=g$. By Corollary \ref{cor:icgroundstate}, $\tilde w (\cdot,t)=e^{-\mu_{s,p} t}g$. 

\par In addition, by (1) and since $g$ is a ground state
$$
\int_\Omega|g|^pdx = \frac{1}{\lambda_{s,p}}[g]_{W^{s,p}(\R^n)}^p=\frac{1}{\lambda_{s,p}}\lim_{j\rightarrow\infty} [v_j(x,0)]_{W^{s,p}(\R^n)}^p\ge \int_\Omega |w|^pdx.\nonumber
$$
However, sending $j\rightarrow\infty$ in \eqref{JustGonnaSendJ} gives 
$$
\int_\Omega |g|^p dx=\int_\Omega |g^+|^pdx\le  \frac{1}{2}\int_\Omega |w|^pdx .
$$
This is a contradiction as $w\not\equiv 0$. 
\end{proof}
\begin{cor}\label{cor:sign} Assume that $v$ is a weak solution of \eqref{pParabolic}. For any positive ground state $w$ for $\laps$ and any constant $C>0$,  there is a $\tilde \delta = \tilde \delta(w,C)>0$ such that if
\begin{enumerate}
\item $\displaystyle e^{\mu_{s,p}pt}[v(\cdot,t)]^p_{W^{s,p}(\R^n)}\geq \lambda_{s,p}\int_\Omega |w|^p $ for all $t\geq 0$, %$\displaystyle\lim_{k\to \infty }e^{\mu_{s,p} s_k}v(x,s_k)=w$ in $W^{s,p}(\R^n)$

\item $\displaystyle\int_\Omega |v(x,0)|^pdx\leq C$,

\item $\displaystyle \frac{[v(\cdot,0)]_{W^{s,p}(\R^n)}^p}{\displaystyle\int_\Omega |v(x,0)|^pdx}\leq \lambda_{s,p}+\tilde \delta$,

\item $\displaystyle\int_\Omega |v^+(x,0)|^pdx\geq\frac{1}{2}\int_\Omega |w|^pdx$,
\end{enumerate}
then
$$
\int_\Omega |e^{\mu_{s,p} t} v^+(x,t)|^pdx\geq \frac{1}{2}\int_\Omega |w|^pdx .
$$
for all $t\geq 0$.
\end{cor}
\begin{proof} Choose
$$
\tilde \delta = \min \left(\lambda_{s,p},\delta(w,2C)\right),
$$
where $\delta(w,2C)$ is from Lemma \ref{lem:sign}. It is clear that $v$ satisfies the assumptions of Lemma \ref{lem:sign}, so that in particular
$$
\int_\Omega |e^{\mu_{s,p} } v^+(x,1)|^pdx\geq \frac{1}{2}\int_\Omega |w|^pdx .
$$
By Proposition \ref{RayleighGoDown} combined with (2) and (3)
$$
\int_\Omega |e^{\mu_{s,p}t} v(x,t)|^p dx \leq\frac{1}{\lambda_{s,p}} [e^{\mu_{s,p}t}v(\cdot,t)]^p_{W^{s,p}(\R^n)}\leq\frac{1}{\lambda_{s,p}}  [v(\cdot,0)]^p_{W^{s,p}(\R^n)}\leq C\left(1+\frac{\tilde \delta}{\lambda_{s,p}}\right)\leq 2C
$$
for any $t>0$. 

\par This inequality, together with (1) and the monotonicity of the Rayleigh quotient, implies that $e^{\mu_{s,p}k} v(x,t+k)$ satisfies properties (1)--(3) in Lemma \ref{lem:sign} with $C$ is replaced by $2C$. In particular, Lemma \ref{lem:sign} applied to $e^{\mu_{s,p}} v(x,t+1)$ yields 
$$
\int_\Omega |e^{2\mu_{s,p} } v^+(x,2)|^pdx\geq \frac{1}{2}\int_\Omega |w|^pdx .
$$
Now we can apply Lemma \ref{lem:sign} repeatedly to $e^{\mu_{s,p} k}v(x,t+k)$ for $k=2,3,\ldots$ in order to obtain the desired result.
\end{proof}
We are now ready to treat the large time behavior of solutions of equation \eqref{pParabolic}.
\begin{proof}[~Proof of Theorem \ref{LargeTthm}]
Let $v$ be a weak solution of \eqref{pParabolic}. By \eqref{eq:needed} in Corollary \ref{GradUNon}
\begin{equation}\label{Cor4u}
\frac{d}{dt} [e^{\mu_{s,p}t}v(\cdot,t)]^p_{W^{s,p}(\R^n)}\le 0
\end{equation}
for almost every $t\ge 0$. Consequently, the limit  
$$
S:=\lim_{t\rightarrow \infty}[e^{\mu_{s,p}t}v(\cdot,t)]^p_{W^{s,p}(\R^n)}
$$ 
exists. If $S=0$, there is nothing else to prove. Let us assume otherwise. 

Let $\tau_k$ be an increasing sequence of positive numbers such that $\tau_k\to \infty$ as $k\rightarrow\infty$, and define for $k=1,2,3,\ldots$
\begin{equation}\label{lastVeeKay}
v^k(x,t)=e^{\mu_{s,p} \tau_k}v(x,t+\tau_k).
\end{equation}
Then $v^k$ is a weak solution of \eqref{pParabolic} with initial data
$$
g^k(x):=e^{\mu_{s,p} \tau_k}v(x,\tau_k). 
$$
By \eqref{Cor4u}, $g^k\in W_0^{s,p}(\Omega)$ is uniformly bounded in $W^{s,p}(\R^n)$. Hence, it is clear that $v^k$ satisfies the hypotheses of Theorem \ref{CompactnessLem}.  Therefore, we can extract a subsequence $\{v^{k_j}\}_{j\in\N}$ converging to a weak solution $w$ as detailed in Theorem \ref{CompactnessLem}.  We may also assume that $v^{k_j}(\cdot,t)$ converges to $w(\cdot,t)$ in $W^{s,p}(\R^n)$ for almost every time $t\ge 0$ since this occurs for a subsequence.  

We now observe that by \eqref{Cor4u}
\begin{equation}\label{Sequation}
S=e^{\mu_{s,p}pt}\lim_{j\rightarrow \infty}[v^{k_j}(\cdot,t)]^p_{W^{s,p}(\R^n)}=e^{\mu_{s,p}tp}[w(\cdot,t)]^p_{W^{s,p}(\R^n)}
\end{equation}
for almost every time $t\ge 0$.  Since $[0,\infty)\ni t\mapsto [w(\cdot,t)]^p_{W^{s,p}(\R^n)}$ is absolutely continuous (by Lemma \ref{AbsPhi}),
$$
S=e^{\mu_{s,p}tp}[w(\cdot,t)]^p_{W^{s,p}(\R^n)}
$$
holds for all $t\ge 0$. As $w$ is a solution of \eqref{pParabolic}, \eqref{eq:needed} in Corollary \ref{GradUNon} implies
$$
0=\frac{1}{p}\frac{d}{dt}\left(e^{\mu_{s,p}tp}[w(\cdot,t)]^p_{W^{s,p}(\R^n)}\right)=e^{(\mu_{s,p}p)t}\left(\mu_{s,p} [w(\cdot,t)]^p_{W^{s,p}(\R^n)}- \int_{\Omega}|w_t(x,t)|^pdx\right)
$$
for almost every $t\ge 0$.

A more careful inspection of the proof of \eqref{Dvvtbound} reveals that if
$$
\mu_{s,p} [w(\cdot,t)]^p_{W^{s,p}(\R^n)}= \int_{\Omega}|w_t(x,t)|^pdx
$$
then we must have $[w(\cdot,t)]^p_{W^{s,p}(\R^n)}=\lambda_{s,p}\int_{\Omega}|w(x,t)|^pdx$.  Therefore $w(\cdot,t)$ is ground state almost every $t>0$. The absolute continuity of $[w(\cdot,t)]_{W^{s,p}(\R^n)}$ and $\|w(\cdot,t)\|_{L^p(\Omega)}$ then implies that $w(\cdot,t)$ is a ground state for all $t\geq 0$. By Corollary \ref{cor:icgroundstate}, 
$$
w(x,t)=e^{-\mu_{s,p}t} w_0,
$$
where $w_0(x)=w(x,0)$ is a ground state. 

For any $t_0\in [0,T]$ such that the limit \eqref{Sequation} holds, we have by Proposition \ref{RayleighGoDown}
\begin{align*}
\lim_{t\rightarrow \infty}\frac{[v(\cdot,t)]^p_{W^{s,p}(\R^n)}}{\int_{\Omega}|v(x,t)|^pdx} &=\lim_{j\rightarrow \infty} \frac{[v(\cdot,\tau_{k_j}+t_0)]^p_{W^{s,p}(\R^n)}}{\int_{\Omega}|v(x,\tau_{k_j}+t_0)|^pdx}\\
&=\lim_{j\rightarrow \infty} \frac{[v^{k_j}(\cdot,t_0)]^p_{W^{s,p}(\R^n)}}{\int_{\Omega}|v^{k_j}(x,t_0)|^pdx}\\
&=\frac{[w(\cdot,t_0)]^p_{W^{s,p}(\R^n)}}{\int_{\Omega}|w(x,t_0)|^pdx}\\
&=\frac{[w_0]^p_{W^{s,p}(\R^n)}}{\int_{\Omega}|w_0|^pdx}\\
&= \lambda_{s,p}.
\end{align*}
Since weak convergence together with the convergence of the norm implies strong convergence, the limit $w_0=\lim_{j\rightarrow\infty}e^{\mu_{s,p}\tau_{k_j}}v(\cdot, \tau_{k_j})$ holds in $W^{s,p}(\R^n)$. We conclude that $\{e^{\mu_{s,p}\tau_k}v(\cdot, \tau_k)\}_{k\in \N}$ has a subsequence converging in $W^{s,p}(\R^n)$ to a ground state $w_0$.

Recall that, due to the simplicity of the ground states, $w_0$ is determined completely by its sign and the constant 
$$
S=[w_0]^p_{W^{s,p}(\R^n)}.
$$
We may assume $w_0\geq 0$, if not we can consider $-v$ instead. 

Now we note that for any $t\geq 0$
\begin{equation}\label{eq:ass1}
e^{\mu_{s,p}pt}[v(\cdot,t)]_{W^{s,p}(\R^n)}^p\geq \lim_{j\to\infty}e^{\mu_{s,p}p\tau_{k_j}}[v(\cdot,\tau_{k_j})]_{W^{s,p}(\R^n)}^p=[w_0]_{W^{s,p}(\R^n)}^p=\lambda_{s,p}\int_\Omega |w_0|^p.
\end{equation}
Since $v^{k_j}(x,0)=g^{k_j}(x)$ converges in $W^{s,p}(\R^n)$ to ground state $w_0$, we also have
\begin{equation}\label{eq:ass2}
\lim_{j\rightarrow \infty} \frac{[v^{k_j}(\cdot,0)]^p_{W^{s,p}(\R^n)}}{\int_{\Omega}|v^{k_j}(x,0)|^pdx}=\lambda_{s,p}
\end{equation}
and
\begin{equation}\label{eq:ass3}
\int_\Omega  |(v^{k_j}(x,0))^+|^p dx \geq \frac12\int_\Omega |w_0|^p dx,
\end{equation}
whenever $j$ is large enough. In addition, due to the monotonicity of the $W^{s,p}$-norm, 
\begin{equation}\label{eq:ass4}
\int_\Omega |v^{k_j}(x,0)|^p dx\leq \frac{1}{\lambda_{s,p}}[v^{k_j}(\cdot,0)]_{W^{s,p}(\R^n)}^p\leq \frac{1}{\lambda_{s,p}}[g]_{W^{s,p}(\R^n)}^p,
\end{equation}
where $g=v(\cdot,0)$.
From \eqref{eq:ass1}--\eqref{eq:ass4}, it is now clear that for $j$ large enough, $v^{k_j}$ satisfies assumptions (1)--(4) in Corollary \ref{cor:sign}, with $w=w_0$ and $C=[g]_{W^{s,p}(\R^n)}^p/\lambda_{s,p}$. As a result, 
$$
\int_\Omega |e^{\mu_{s,p} t} (v^{k_j})^+(x,t)|^pdx\geq \frac{1}{2}\int_\Omega |w_0|^pdx
$$
for all $t\geq 0$, which implies
\begin{equation}
\label{eq:signcontrol}
\int_\Omega |e^{\mu_{s,p} t} (v)^+(x,t)|^pdx\geq \frac{1}{2}\int_\Omega |w_0|^pdx
\end{equation}
for $t$ large enough.

Suppose now that we pick another convergent subsequence of $\{e^{\mu_{s,p}\tau_k}v(\cdot, \tau_k)\}_{k\in \N}$. Then arguing as above, the sequence converges in $W^{s,p}(\R^n)$ to a ground state $w_1$ and by \eqref{Sequation}, 
$$
[w_1]^p_{W^{s,p}(\R^n)}=S.
$$
By the simplicity of ground states, $w_1=w_0$ or $w_1=-w_0$. Passing $t\to \infty$ in \eqref{eq:signcontrol}, we obtain 
$$
\int_\Omega (w_1^+)^p dx \geq \frac12\int_\Omega |w_0|^p dx, 
$$
forcing $w_1$ to be positive and hence $w_1=w_0$. As the sequence
$\{\tau_k\}_{k\in \N}$ was chosen arbitrarily, it follows that $e^{\mu_{s,p} t}v(\cdot, t)\rightarrow w_0$ as $t\rightarrow \infty$ in $W^{s,p}(\R^n)$. 
\end{proof}
We are finally in a position to prove Theorem \ref{LargeTthmvisc}.
\noindent \begin{proof}[~Proof of Theorem \ref{LargeTthmvisc}] Let $\tau_k$ be an increasing sequence of positive numbers such that $\tau_k\to \infty$ as $k\rightarrow\infty$. In Theorem \ref{LargeTthm}, we established that $\lim_{k\rightarrow \infty}e^{\mu_{s,p} \tau_k}v(\cdot,\tau_k)=w$ in $W^{s,p}(\R^n)$.  In view of Proposition \ref{ViscSolnResult}, it suffices to show that this convergence occurs uniformly on $\overline{\Omega}$. 

% that the $v^k$ converges occurs uniformly in $\overline{\Omega}\times[0,1]$ to $e^{-\mu_{s,p} t}w(x)$. Here $w$ is the limit from Theorem

\par To this end, define $v^k$ as in \eqref{lastVeeKay}. We also remark that $e^{-\mu_{s,p} t}\Psi $ is a solution of equation \eqref{MainPDE}.
By the comparison principle, 
\begin{equation}\label{anotherVeekayEst}
|v^k(x,t)|\leq \Psi(x)
\end{equation}
for $(x,t)\in \R^n\times [-1,1]$ for all $k\in \N$ large enough. These bounds with Theorem \ref{Holderthm} give that $v^k$ is uniformly bounded in $C^\alpha(B\times [0,1])$ for any ball $B\subset\subset \Omega$. By a routine covering argument, $v^k$ is then uniformly bounded in $C^\alpha(K\times [0,1])$ for any compact $K\subset\subset \Omega$. We now claim that the sequence $v^k$ is equicontinuous in $\overline \Omega\times [0,1]$. 

Fix $\e>0$. Recall that $\Psi$ is continuous up to the boundary of $\Omega$. By \eqref{anotherVeekayEst}, it follows that there is $\delta$ so that 
$|v^k(x,t)|\leq \e/2$ whenever $d(x):=d(x,\partial\Omega)<\delta$ and $t\in [0,1]$. Now we will show that if $|x-y|+|t-\tau|$ is small enough, then $|v^k(x,t)-v^k(y,\tau)|\leq \e$. We treat three cases differently as follows.
\begin{enumerate}
\item $ d(x)<\delta/2$ and $d(y)<\delta/2$: Then 
$$|v^k(x,t)-v^k(y,\tau)|\leq |v^k(x,t)|+|v^k(y,\tau)|\leq \e.$$
\item $d(x)<\delta/2$ and $d(y)>\delta/2$: Then $|x-y|<\delta/2$ implies $d(y)<\delta$ so that again 
$$|v^k(x,t)-v^k(y,\tau)|\leq |v^k(x,t)|+|v^k(y,\tau)|\leq \e.$$
\item $d(x)>\delta/2$ and  $d(y)>\delta/2$: Then by the local H\"older estimates 
$$
|v^k(x,t)-v^k(y,\tau)|\leq C_\delta \left(|x-y|+|t-\tau|\right)^\alpha\leq \e
$$
if we choose $|x-y|+|t-\tau|\leq (\e/C_\delta)^{1/\alpha}$.
\end{enumerate}
Hence, if
$$
|x-y|+|t-\tau|\leq \min\left(\delta/2, (\e/C_\delta)^{1/\alpha}\right)
$$
then $|v^k(x,t)-v^k(y,\tau)|\leq \e$. Therefore, the sequence $v^k$ is equicontinuous on $\overline\Omega \times [0,1]$. By the Arzel\`a-Ascoli theorem, we can extract a subsequence $v^{k_j}$ converging to $e^{-\mu_{s,p}t}w$, the limit in \eqref{FullTime}, uniformly in $\overline\Omega\times [0,1]$. 
\end{proof}

\bibliographystyle{amsrefs}
\bibliography{refnonlocal}

\end{document}